\documentclass[11pt, a4paper]{amsart}
 
\usepackage[euler-digits]{eulervm}

\usepackage{amsfonts, amsthm, amssymb, amsmath, stmaryrd}
\usepackage{mathrsfs,array}
\usepackage{eucal,fullpage,times,color,enumerate,accents}
\usepackage{url}

\usepackage{hyperref}
\hypersetup{
  colorlinks   = true,          
  urlcolor     = blue,          
  linkcolor    = teal,          
  citecolor   = orange             
}

\usepackage{color}
\usepackage{mathrsfs}
\usepackage{amssymb}
\usepackage{tikz}
\usepackage{tikz-cd}
\usepackage{bm}
\usepackage{enumerate}
\usetikzlibrary{calc}
\usepackage{subfigure}

\title{Tropical compactification and the Gromov--Witten theory of $\mathbb{P}^1$}
\author{Renzo Cavalieri, Hannah Markwig, and Dhruv Ranganathan}
\date{\today}

\address{Department of Mathematics, Colorado State University}
\email{renzo@math.colostate.edu}

\address{Eberhard Karls Universit\"at T\"ubingen, Fachbereich Mathematik, Institut f\"ur Geometrie}
\email{hannah@math.uni-tuebingen.de}

\address{Department of Mathematics, Yale University}
\email{dhruv.ranganathan@yale.edu}

\newtheorem{theorem}{Theorem}
\newtheorem{corollary}[theorem]{Corollary}
\newtheorem{lemma}[theorem]{Lemma}
\newtheorem{proposition}[theorem]{Proposition}
\newtheorem{definition}[theorem]{Definition}

\newtheorem{construction}[theorem]{Construction}
\newtheorem{quasi-theorem}[theorem]{Quasi-Theorem}

\newtheorem{rem1}[theorem]{Remark}
\newenvironment{remark}{\begin{rem1}\em}{\end{rem1}}

\newtheorem{ex1}[theorem]{Example}
\newenvironment{example}{\begin{ex1}\em}{\end{ex1}}

\newtheorem{not1}[theorem]{Notation}

\newcommand{\A}{{\mathbb{A}}}           
\newcommand{\CC} {{\mathbb C}}

\newcommand{\PP}{\mathbb{P}}         
		
\newcommand{\RR} {{\mathbb R}}		
\newcommand{\ZZ} {{\mathbb Z}}

\newcommand{\trop}{t\!r\!o\!p}
\newcommand{\an}{a\!n}
\newcommand{\rub}{r\!u\!b}
\newcommand{\parm}{p\!a\!r}

\newcommand{\Mbar}{\overline{M}}

\begin{document}

\pagestyle{plain}
\maketitle

\begin{abstract}
We use tropical and nonarchimedean geometry to study the moduli space of genus $0$, stable maps to $\mathbb{P}^1$ relative to two points. This space is exhibited as a tropical compactification in a toric variety. Moreover, the fan of this toric variety may be interpreted as a moduli space for tropical relative stable maps with the same discrete data. As a consequence, we confirm an expectation of Bertram and the first two authors, that the tropical Hurwitz cycles are tropicalizations of classical Hurwitz cycles. As a second application, we obtain a full descendant correspondence for genus $0$ relative invariants of $\mathbb{P}^1$.
\end{abstract}

\setcounter{tocdepth}{1}
\tableofcontents
\pagebreak
\setcounter{section}{-1}

\section{Introduction}

\subsection{Main results} Let $\bm x = (x_1,\ldots, x_n) \in \ZZ^n$ be a collection of  non-zero integers satisfying 
$
\sum x_i = 0. 
$
Denote by $\bm x^+$ and $\bm x^-$ the positive and negative entries of $\bm x$. Let $\overline M(\bm x)$ be the coarse moduli space of rational rubber stable maps to $\PP^1$, relative to $0$ and $\infty$, having ramification over $0$ (resp. $\infty$) given by the entries in $\bm x^+$ (resp. $\bm x^-$). We mark the special ramification locus on the source curve. Let $M(\bm x)$ be the open set parametrizing maps from a smooth curve. We review the basics of relative stable maps in Section~\ref{sec: rsm-background}. 

Our main result identifies the space $\overline M(\bm x)$ as the closure of its open part in a toric variety, naturally associated to the space of tropical relative stable maps with ramification data given by $\bm x$. 

\begin{theorem}\label{thm: tropical-compactification}
There is a simplicial (noncomplete) toric variety $X(\Delta^{\rub}_{\bm x})$ with dense torus $T$, and an embedding $M(\bm x)\hookrightarrow T$, such that the closure of $M(\bm x)$ in $X(\Delta^{\rub}_{\bm x})$ is identified with the coarse moduli space $\overline M(\bm x)$. The fan $\Delta^{\rub}_{\bm x}$ of the toric variety is naturally identified with the tropical moduli space $\overline M^{\trop}(\bm x)$. 
The compactification of $M(\bm x)$ is sch\"on, meaning that  the intersection of $\overline M(\bm x)$ with each torus orbit is smooth. 
\end{theorem}


We obtain two main consequences. The first is a correspondence theorem for the relative descendant Gromov--Witten theory of $\PP^1$. Denote by $\overline M_r(\PP^1,\bm x)$ the space of relative stable maps to a parametrized target $\PP^1$, with ramification over $0$ and $\infty$ given by $\bm x$, with $r$ additional non-relative markings. We use correlator notation, defining
\[
\langle \tau_{k_1}(pt),\ldots, \tau_{k_s}(pt),\tau_{k_{s+1}}(1),\ldots,\tau_{k_r}(1)\rangle = \int_{[\overline M_r(\PP^1,\bm x)]} \prod_{i=1}^s \hat \psi_i^{k_i}ev_i^\star(pt)\prod_{i=s+1}^r\hat \psi_i^{k_i}.
\]
Here $\hat \psi$ are the descendant classes, defined on the space of maps~\cite{OP06}. The tropical invariants are defined analogously, see, for example~\cite{MR08}.

\begin{theorem}\label{thm: descendant-correspondence}
We have an equality of classical and tropical genus $0$ relative descendant Gromov--Witten invariants of $\PP^1$:
\[
\langle \tau_{k_1}(pt),\ldots, \tau_{k_s}(pt)\rangle_{\PP^1}^0=
\langle \tau_{k_1}(pt),\ldots, \tau_{k_s}(pt)\rangle^0_{\PP^1,\trop}.
\]
\end{theorem}

The second consequence is exhibiting the tropical Hurwitz loci as a tropicalization. Fix $r$ to be the number of simple branch points of a generic cover with special ramification $\bm x$, and fix a positive integer $k<r$. Given a collection of points $\underline p = (p_1,\ldots, p_{r-k})\in (\RR)^{r-k}\subset (\PP^1_{\trop})^{r-k}$, one may construct a $k$-dimensional tropical Hurwitz locus $\mathbb H^{trop}_k(\underline p;\bm x)$,~\cite{BCM}. Hampe~\cite{H14} shows that such tropical Hurwitz cycles exhibit many of the combinatorial properties expected of a tropicalization, such as being connected through codimension $1$.   The following result confirms this expectation.
\begin{theorem}\label{thm: hurwitz-correspodnence}
The tropical Hurwitz loci $\mathbb H_k^{\trop}(\underline p;\bm x)$ are tropicalizations of  loci $H_k(\underline y;\bm x)\subset \overline M_{0,n}(K)$ where $K$ is a algebraically closed valued field extending $\CC$. The locus $H_k(\underline y;\bm x)$ represents a Hurwitz locus in $\overline M_{0,n}(K)$. 
\end{theorem}

Theorem~\ref{thm: hurwitz-correspodnence} gives us control of the basic combinatorial structure of the Hurwitz loci.

\begin{corollary}
The loci $\mathbb H_k^{\trop}(\underline p;\bm x)$ are connected through codimension $1$, and may be given the structure of a balanced polyhedral cone complex. Moreover, the tropical Hurwitz loci for different choices of $p_i$ are tropically rationally equivalent. 
\end{corollary}

The definitions of the tropical and classical Hurwitz loci are recalled in Section~\ref{sec: hurwitz-loci}. 

\subsection{Context and Motivation} The current work contributes to the authors' long term program  of exploring the tight interconnection between the tropical and classical theories of moduli spaces of maps to curves. The two highlights of this work are the following: first, that the combinatorics of the tropical moduli spaces reflects and informs the birational geometry of classical ones. Second, that several ``geometric" tautological classes, including evaluation cycles, psi classes, and Hurwitz loci, tropicalize in a natural way, exhibiting a solid conceptual connection between the classical and tropical intersection theories. When the source curves have genus $0$, the (open) moduli space of maps can be embedded in a torus. This makes available standard techniques from tropical compactification~\cite{Tev07} and toric intersection theory~\cite{FS97,Kat09}. In higher genus, the natural framework for tropicalization exploits the toroidal structure of these moduli spaces. Tropical intersection theory remains undeveloped in this setting.

\subsubsection{Relative stable maps as birational models for $\overline M_{0,n}$}The spaces $\overline M(\bm x)$ of rubber relative stable maps are birational to $\overline M_{0,n}$. The natural forgetful morphism
\[
st: \overline M(\bm x)\to \overline M_{0,n},
\]
is an isomorphism over the open set $M_{0,n}$. Theorem~\ref{thm: tropical-compactification} may be interpreted as stating that tropical geometry completely determines these birational models. This statement is made precise in the following manner. Let $T$ be a torus over $\CC$, endowed with the trivial valuation, with one-parameter subgroup lattice $N$. Given any subvariety $Y \hookrightarrow T$, we obtain a set $trop(Y)\subset N_\RR$. Choose a fan $\Delta$ in $N_\RR$, such that $trop(Y)$ is supported on $\Delta$. For every fan $\Delta$, we obtain a compactification $\overline Y$ of $Y$, by taking the closure of $Y$ in the toric variety $X(\Delta)$~\cite{Tev07}. The condition of $trop(Y)$ being supported on the cones of $\Delta$ ensure that $\overline Y$ intersects only those torus orbits of complementary dimension. 

Given ramification data $\bm x$, we obtain compactifications of the moduli space $M_{0,n}$ following the above procedure. We may embed the very affine variety $M_{0,n}$ in a torus $T$ using the procedure described by Kapranov, which we recall in Section~\ref{sec: trop-compactification-rsm}. As above, fixing a fan structure on the tropicalization $trop(M_{0,n})$,  we obtain compactifications of $M_{0,n}$. In particular, we may choose the fan structure to be the tropical moduli space $M^{\trop}(\bm x)$. Theorem~\ref{thm: tropical-compactification} is then asserting that the compactification obtained via this fan structure is precisely the space $\overline M(\bm x)$ of rubber maps with ramification data $\bm x$. 

Fixing the length $n$ of the tuple $\bm x$ and letting the entries vary along the sum-zero sublattice one obtains a family of birational models of $\overline{M}_{0,n}$, such that the dual intersection complex of the toroidal boundary is naturally identified with the tropical moduli space; such complexes remain combinatorially equivalent in chambers defined by the resonance hyperplanes. This  corroborates the expectation that the piecewise polynomiality of Hurwitz cycles \cite{BCM} and of Hurwitz numbers \cite{GJV} arises from the variation of combinatorial structure of the toroidal boundary of the relevant moduli spaces.



\subsubsection{Tropical computations of tautological intersection numbers} There has been a great deal of progress in the last decade concerning correspondence theorems for intersection numbers on moduli spaces of curves and maps, since Mikhalkin's celebrated results~\cite{Mi03} on curve counting in toric surfaces. 
The current work  enlarges the scope of the classical/tropical intersection theory correspondence for moduli spaces of maps with target curves, from Hurwitz numbers \cite{BBM,CJM1, CMR14} to descendant relative Gromov-Witten invariants of $\PP^1$.

Besides this, an interesting feature of the current work is that the techniques employed seem to be suitable for being generalized and applied to a broader context.
Recent work of A.\ Gross~\cite{Gro14} studies  invariants arising from moduli spaces of rational curves in arbitrary toric varieties, with incidence conditions defined by evaluation morphisms. 
The approach is to embed the open part of the moduli space of maps into a torus, without compactification, and obtain correspondence by applying lifting theorems~\cite{OP,OR} to tropical intersections. A feature of his work is that the embedding into a torus allows him to work with non-complete moduli spaces. 
Rau and the second author~\cite{MR08}, and independently M.\ Gross~\cite{Gro10} and Overholsher~\cite{Over15}, proved that a certain sector of the genus $0$ descendant theory of $\PP^2$ may be computed on the tropical side. However, the result for $\PP^2$ remains incomplete, and in particular, a correspondence for genus $0$ (non-logarithmic) invariants with one descendant insertion that is not paired with a point class is missing. We hope that progress for target $\PP^1$ will lend insight into this problem. 

In dimension $2$ and higher, the toric boundary is not a smooth divisor, and there do not exist relative theories based on expanded degenerations of the target along the toric boundary. Instead, one might appeal to the recently developed framework of logarithmic stable maps~\cite{AC11, Che10, GS13}. For $\PP^1$, the logarithmic and relative theories are closely related~\cite[Appendix B]{AMW12}, and thus, we view a rigorous study of tropical descendants in $\PP^1$ as the natural first step in completing the descendant correspondence for $\PP^2$, developing tropical computations of logarithmic invariants, and further generalizations thereof. 

\subsection{Recent progress and future directions} Since this preprint appeared on the arXiv, there has been additional progress in this area of study. In~\cite{R15b}, the third author extends the results in this paper to logarithmic stable maps from genus $0$ curves to toric targets, recovering the geometric description of the space of maps as a modification of $\overline M_{0,n}$. Using these ideas, A. Gross extended the computation of logarithmic descendants to this setting~\cite{Gro15}. 


A natural generalization is to study spaces of relative stable maps from higher genus curves to toric targets. Even when the target is $\PP^1$, such spaces can be quite badly behaved -- they are singular, reducible, and nonequidimensional. Establishing correspondence theorems for higher genus invariants requires the development of virtual cycles in the tropical setting, which is at the same time a daunting and exciting task.

\subsection*{Acknowledgements} RC is grateful for the support from  NSF
grant DMS-1101549, NSF RTG grant 1159964. HM was partially supported by DFG-grant MA 4797/6-1. During the preparation of this work, we benefited from helpful conversations with friends and colleagues, including Dan Abramovich, Aaron Bertram, Noah Giansiracusa, Mark Gross, Dave Jensen, Diane Maclagan, Martin Ulirsch and Jonathan Wise. The third author thanks Sam Payne in particular for his constant encouragement and numerous insightful conversations regarding many ideas central to this work. We would like to thank an anonymous referee for helpful comments on an earlier version of this paper.


\section{Preliminaries}\label{sec-prelim}

The methods used to prove many of the results in this paper require standard facts from the theory of toric varieties, geometric tropicalization, and toroidal embeddings. We provide a rapid outline of the relevant concepts, leaving further details to the literature. The reader is referred to the survey~\cite{Pay15} for an introduction to Berkovich spaces and to~\cite{ACMUW} and~\cite{U13} for an overview of the theory of skeletons. 

We work over the field $\CC$ topologized by the trivial valuation. Let $T$ be a torus with character lattice $M$. Let $\Delta$ be a fan in $N_\RR = Hom(M,\RR)$ and let $X = X(\Delta)$ be the associated toric variety~\cite{Ful93}. The Berkovich analytification (hereafter simply \textit{analytification}) of the dense torus $T^{an}$ admits a continuous \textit{tropicalization} map
\[
trop: T^{an}\to N_\RR,
\]
which extends to the map 
\[
trop: X^{an}\to \overline \Delta,
\]
referred to as the extended tropicalization map~\cite{Pay09}. The target $\overline \Delta$ is a compactification of the vector space $N_\RR$ that is stratified in a manner reflecting the stratification of $X$ into torus orbits. More precisely, $\overline \Delta$ is the canonical compactification of the fan $\Delta$ obtained as follows. Let $\sigma$ be a cone of $\Delta$, and let $\sigma^{\vee}$ be the dual cone, and $S_\sigma = \sigma^\vee \cap M$. Recall that $\sigma$ can be recovered from $S_\sigma$ as the space of monoid homomorphisms
\[
\sigma = Hom_{\textnormal{mon}}(S_\sigma, \RR_{\geq 0}).
\]
Correspondingly, the compactification $\overline \sigma$ is obtained as the space of monoid homomorphisms
\[
\overline \sigma = Hom_{\textnormal{mon}}(S_\sigma,\RR_{\geq 0}\sqcup \{\infty\}). 
\]
The extended cone $\overline \sigma$ is compact and contains $\sigma$ as a dense open subset. The \textit{extended cone complex} $\overline \Delta$ is then obtained by gluing these extended cones $\overline \sigma$ in the natural way. 

Given a subvariety $Y\hookrightarrow X$, there is a map 
\[
trop: Y^{an}\to \overline \Delta,
\]
obtained by restricting $trop$ via $Y^{an}\hookrightarrow X^{an}$. The notation $trop(Y)$ will be used as shorthand for $trop(Y^{an})$. Since we work in the case where the ground field is trivially valued, the set $trop(Y\cap T)$ inherits the structure of a balanced fan\footnote{If we work over a general nonarchimedean field $K$, this set will carry the structure of a balanced polyhedral complex.}~\cite{MS14}.

Let $Y$ be an $\ell$-dimensional subvariety of $X(\Delta)$. Then $Y$ is said to \textit{intersect the toric boundary properly} if and only if for all orbit closure $V(\tau)$ in $X$, $\dim(Y\cap V(\tau)) = \ell+\dim(V(\tau))-n$. It is a standard fact~\cite{Gub13} that $Y$ intersects the toric boundary properly if and only if $trop((Y\cap T)^{an})$ is a union of cones of $\Delta$. 

We are interested in \textit{tropical compactification}. Let $T$ be a torus, and let $Y$ be a subvariety of $T$. Fix a fan structure $\Delta$ on $trop(Y)$, and let $X(\Delta)$ be the associated toric variety, and $\overline Y$ the closure of $Y$ in $X(\Delta)$. If $\overline Y$ is proper and the multiplication map 
\[
\overline Y\times T\to X(\Delta)
\]
is faithfully flat, we refer to $\overline Y$ as a \textit{tropical comapctification} of $Y$. Moreover, such a  compactification is called sch\"on if the intersection of $\overline Y$ with each torus orbit is smooth. The intersection theory of sch\"on tropical compactifications is closely related to the tropical intersection theory of the fan $\Delta$. In particular, in many cases, degrees of zero cycles on $\overline Y$ may be computed from the tropical intersection theory of $\Delta$~\cite{Kat09}. 

The tropicalization map on a torus or toric variety depends on a global choice of monomial coordinates given by the character lattice of the torus. More generally, one may ask for such a map when there is a consistent choice of \textit{local} monomial coordinates, in an appropriate sense. This is formalized by work of Thuillier~\cite{Thu07} and Abramovich, Caporaso, and Payne~\cite{ACP}. A convenient and flexible framework that includes toroidal embeddings is that of logarithmic structures~\cite{U13}. The toroidal case suffices for our purposes. 

A toroidal embedding (without self-intersection) is a normal variety $X$, together with an open embedding $U\hookrightarrow X$ which locally looks like the inclusion of the dense torus into a toric variety~\cite{KKMSD}. More precisely, it is the data of  $U\hookrightarrow X$ such that at every point $p\in X$, there exists a Zariski open neighborhood $V$ of $p$, and an \'etale map $\psi:V\to V_\sigma$ to an affine toric variety, such that 
\[
V\cap U = \psi^{-1}T,
\]
where $T$ is the dense torus in $V_\sigma$. Analogously, one may define a toroidal Deligne--Mumford stack~\cite{ACP}. The cones $\sigma$ associated to the local models can be glued  to form a polyhedral cone complex $\Sigma(X)$, and the canonical compactifications $\overline \sigma$ can be glued to form an extended cone complex $\overline \Sigma(X)$. There exists a tropicalization map
\[
trop:X^{an}\to \overline \Sigma(X),
\]
sending $U^{an}$ to $\Sigma(X)$. This map is often also referred to as the \textit{projection to the Thuillier skeleton.} Every toric variety is of course a toroidal embedding, and it is an elementary fact that the two notions of tropicalizations coincide. That is, if $X = X(\Delta)$ is a toric variety, then in the above notation, $\overline \Sigma(X)=\overline \Delta$. The substantial difference between the polyhedral complexes arising as skeletons of toroidal embeddings is that these polyhedral complexes do not come with a natural embedding in a vector space. Finally, there exists a canonical continuous section $s:\overline\Sigma(X)\hookrightarrow X^{an}$, realizing the cone complex as a closed topological subspace of the analytification~\cite{Thu07}. 

Examples of toroidal embeddings without self-intersection include the moduli spaces $\overline M_{0,n}$ of $n$-pointed rational curves, the Kontsevich spaces $\overline M_{0,n}(\PP^r,d)$, and the spaces of relative stable maps to $\PP^1$ considered in this paper. In each case, the toroidal structure is given by the complement of a divisor with (stack theoretic) strict normal crossings. 

\section{Relative stable maps}\label{sec: rsm-background}

To begin, we briefly recall the theory of relative stable maps to  $\PP^1$. We consider separately the cases where $\PP^1$ is parametrized, and where two maps are considered equivalent when they differ by post-composition by the torus action  on $\PP^1$ fixing $0$ and $\infty$. The latter is often referred to as the case of \textit{rubber target}. For a more detailed account we refer to~\cite[Section 5]{Vak08} for maps to parametrized $\PP^1$ and to~\cite[Section 1.5]{MP06} for  rubber maps. 

We fix a tuple of integers $\bm x\in \ZZ^n$, with $\sum x_i=0$. We separate the positive and negative parts of this tuple, to obtain tuples $\bm x^+$ and $\bm x^-$, and set $d$ to be the sum of the positive entries. We are interested in genus $0$ covers of genus $0$ curves with ramification over $0$ and $\infty$ specified by $\bm x^+$ and $\bm x^-$, respectively. \label{x}

\subsection{Maps to a parametrized target.}\label{subsec-classicalparam} We focus on the theory of genus $0$ relative stable maps to $\PP^1$, relative to two points $0$ and $\infty$. 

A genus $0$ map to $\PP^1$ relative to $0$ and $\infty$, denoted $[f:C\to \PP^1]$ as shorthand, consists of the following data:
\begin{enumerate}[({A}1)]
\item \textbf{Map to the expansion.} A map $f_1$ from a genus $0$, Deligne--Mumford semi-stable, $n$-pointed curve $(C,p_1,\ldots, p_n)$ to a chain of rational curves 
\[
T = T_{-k}\cup T_{-k+1}\cup \cdots \cup T_0 \cup T_1 \cup \cdots \cup T_\ell,
\]
where $T_i$ and $T_{i+1}$ meet, with distinguished points $\infty_T\in T_{-k}-T_{-k+1}$ and $0_T\in T_{\ell}-T_{\ell-1}$. 
\item \textbf{Contraction to the parametrized component.} A contraction map $\mathfrak c:T\to \PP^1$, contracting $T_{-i}$ to $0$ and $T_i$ to $\infty$, for $i>0$, and furnishing an isomorphism
\[
(T_0, T_0\cap T_1, T_0\cap T_{-1}) \to (\PP^1,0,\infty).
\]
\item \textbf{Ramification.} For $x_i$ positive (resp. negative) the map $f_1$ is ramified over $0_T$ (resp. $\infty_T$) at the point $p_i$, with index $|x_i|$. 
\item \textbf{Predeformability.} The preimage of each node of $T$ is a union of nodes of $C$. Furthermore, if $q$ is a node of $C$ and we lift $f$ to the normalization of source and target, the ramification orders at the two shadows of the node $q$ are equal. 
\end{enumerate}

An isomorphism of two relative maps is a commuting diagram
\[
\begin{tikzcd}
(C,p_1,\ldots, p_n) \arrow{rr}{\cong} \arrow[swap]{d}{f_1} & &(C',p'_1,\ldots, p'_n) \arrow{d}{f_1} \\
(T,0_T,\infty_T)\arrow[swap]{rr}{\cong} \arrow{dr} & &(T,0_T,\infty_T)\arrow{dl}\\
& (\PP^1,0,\infty) & \\
\end{tikzcd}
\]
where the horizontal arrows are isomorphisms

For every component $T_i$ of the rational chain $T$, we denote by $0$ and $\infty$ the two points $T_i\cap T_{i\pm1}$ where $T_i$ connects to the adjacent components, thus implicitly choosing local coordinates which identify the automorphims of $T_i$ that preserve the distinguished points with the standard torus action on the projective line.
\begin{definition}
We say that a component $C_j$ of the source curve $C$ is a \textit{trivial bubble} if it maps to a component $T_j$ of the expanded target $T$ as a degree $d$ map, fully ramified over $0$ and $\infty$ and unramified elsewhere, i.e. the map is of the form $[z_0:z_1]\mapsto [z_0^d:z_1^d]$. 
\end{definition}

\begin{definition}
A relative map $[f:C\to \PP^1]$ is said to be \textit{stable} if $f_1$-contracted genus $0$ components have at least $3$ special points, and for $i\neq 0$, there exists a component $C'$ of $C$ mapping to $T$ that is not a trivial bubble. 
\end{definition}

\subsection{Maps to a rubber target}\label{subsec-classicalrubber}
We also consider maps to $\PP^1$, with specified ramification over $0$ and $\infty$, but where the target is unparametrized: two maps to $\PP^1$ that differ by a multiplicative constant are considered equivalent.  These spaces arise naturally when analyzing contributions at the fixed locus to the localization formula, applied to the parametrized spaces~\cite{MP06}. A relevant geometric feature is that for these spaces the stabilization morphism has zero dimensional fibers.

Formally, a relative map to a rubber $(\PP^1,0,\infty)$, is a relative map in the definition above, without the data of a contraction to a parametrized component. Furthermore, an isomorphism of relative maps to a rubber $(\PP^1,0,\infty)$ is a commuting diagram,
\[
\begin{tikzcd}
(C,p_1,\ldots, p_n) \arrow{rr}{\cong} \arrow{d} & &(C',p'_1,\ldots, p'_n) \arrow{d} \\
(T,0_T,\infty_T)\arrow[swap]{rr}{\cong} & &(T,0_T,\infty_T)\\
\end{tikzcd}
\]
where horizontal arrows are isomorphisms. We denote maps to rubber targets by $[C\to T]$, since there isn't a privileged component of the rational chain to yield a well defined contraction morphism. The lack of a parametrized component implies that in order to have finite stabilizers, there cannot be components of the expanded target which are mapped to only by trivial bubbles.

\begin{definition}
A relative map to a rubber $(\PP^1,0,\infty)$ is said to be stable if contracted components have at least $3$ special points, and for any component $T_i$ of the expanded target, there exists a component $C'$ of $C$ mapping to $T_i$ that is not a trivial bubble. 
\end{definition}

\subsection{Moduli spaces and tautological morphisms}\label{subsec-classicalmoduli}
Both for a parametrized and a rubber target, there exist moduli spaces parametrizing relative stable maps with ramification prescribed by $\bm x$. It is also useful to mark not only the ramification points over $0_T$ and $\infty_T$, but also additional points. We 
denote by $\Mbar_r(\PP^1,\bm x)$ the space of isomorphism classes of $n+r$-marked stable maps to a parametrized $\PP^1$, relative to $0$ and $\infty$, with ramification data given by $\bm x$. By convention, the first $n$ marked points are the ramification points over $0_T$ and $\infty_T$ (\textit{relative markings}). We denote the corresponding space of rubber maps $\Mbar_r(\bm x)$. If there are no additional markings, i.e.\ $r=0$, we just drop the subscript and write $\Mbar(\PP^1,\bm x)$ resp.\ $\Mbar(\bm x)$ for the rubber space.

There is always a forgetful morphism from the space of maps to a parametrized target to the rubber space. 
\[
\epsilon: \Mbar_r(\PP^1,\bm x)\to \Mbar_r(\bm x).
\]
This map forgets the contraction to the main component, and stabilizes covers where only trivial bubbles lie above the main component. 

When the target is parametrized, there are natural evaluation morphisms to the main component
\begin{eqnarray*}
ev_i: \Mbar_r(\PP^1,\bm x)&\to& \PP^1 \\
{[f: (C,p_1,\ldots,p_r)\to T\to \PP^1]}&\mapsto& f(p_i).
\end{eqnarray*}

\label{tautologicalmorphisms}

There are two additional tautological morphisms on these moduli spaces. We introduce these on the rubber space, where they are better suited to our purposes. Firstly, we have a stabilization morphism 
\[
st:\Mbar_r(\bm x)\to \overline M_{0,n+r}, 
\]
taking a map $[C\to T]$ to its marked source curve $C$. Secondly, we have a branch morphism to (a quotient) of the Losev--Manin space, 
\[
br: \Mbar(\bm x)\to [\overline M_{0,2+(n-2)\cdot \varepsilon}/S_{n-2}]
\]
sending a map $[C\to T]$ to its base curve $T$, with heavy marks at $0_T$ and $\infty_T$, and light marks with appropriate multiplicities yielding the remaining  of the branch divisor for the map. 
We will refine the following proposition in the course of the main result. 

\begin{proposition}\label{prop-stab}
The stabilization morphism $st:\overline M_r(\bm x)\to \overline M_{0,n+r}$ is birational, and an isomorphism when restricted to the open set $M_r(\bm x)$ consisting of maps from smooth curves.
\end{proposition}


\section{Tropicalization for relative stable maps}

\subsection{Tropical relative stable maps} 
\subsubsection{Parametrized target}
The role of a parametrized target in tropical geometry is played by the extended cone complex $\PP^1_{\trop} = \RR\sqcup\{\pm \infty\}$, taken with its usual compactified fan structure. In particular, the cone $0\in \PP^1_{\trop}$ is distinguished, and plays the role of the main component on the classical side. This coincides with the extended tropicalization $trop(\PP^1)$. The expansion of the target on the classical side manifests on the tropical side as a polyhedral subdivision of $\PP^1_{\trop}$. 

Following standard conventions, an \textit{$n$-marked genus $0$ tropical curve} $\Gamma$ is a metrized dual graph of a marked rational nodal curve, where half edges corresponding to marks are metrized as $[0,\infty]$. That is, $\Gamma$ is a metric tree with finitely many leaf edges. Each leaf edge is metrized as the extended interval $[0,\infty]$. 
If the marked rational nodal curve is stable, we obtain a graph whose interior vertices are at least $3$-valent. 
The combinatorial type of an $n$-marked genus $0$ tropical curve is obtained by dropping the metrization data. The stabilization of (the combinatorial type of) an $n$-marked genus $0$ tropical curve is obtained by forgetting $2$-valent vertices (and adding the length of the two adjacent edges).

We continue to fix $\bm x $, a sum-zero vector of integers having positive and negative parts $\bm x^+$ and $\bm x^-$ respectively. We recall the following definition from~\cite{CJM1}. 

\begin{definition}\label{def-trop-parametrized}
A \textit{tropical relative stable map to parametrized $\PP^1_{\trop}$ with relative condition $\bm x$} is an $n$-marked genus $0$ tropical curve $\Gamma$, together with a balanced (or, harmonic) map $h: \Gamma\to \PP^1_{\trop}$ satisfying the following conditions.
\begin{enumerate}[(P1)]
\item The image of $\Gamma$ without its $1$-valent vertices lies inside $\RR$.
\item The map $h$ is piecewise linear, and linear on each edge $e\in \Gamma$, with nonnegative integer expansion factor $w(e)\in \ZZ_{\geq 0}$. 
\item If $x_i$ is positive (resp. negative), $h$ maps the marked end $p_i$
with expansion factor $x_i$ (resp. $-x_i$), and with $+\infty$ (resp. $-\infty$) belonging to the image of the end. 
\end{enumerate}
The full inverse image $h^{-1}(0)$ is (by convention) included as part of the vertex set of $\Gamma$. 
\end{definition}

The data obtained from a tropical relative stable map after dropping all metrization data on the source curve is referred to as a \textit{combinatorial type}. This includes the source graph $\Gamma$, whose vertex set includes all preimages of the main component, an ordering of the vertices according to their images in $\PP^1_{\trop}$, and all expansion factors on edges of $\Gamma$. The combinatorial data on the target is equivalent to that of a line graph, with $2$ infinite edges, and a distinguished vertex $0$.

\begin{remark}
A tropical stable map induces a polyhedral subdivision of the target. Given such a map $[\Gamma\to \PP^1_{\trop}]$, the image of the vertices of $\Gamma$ of valence larger than $2$ yields a finite collection of points on $\PP^1_{\trop}$. Following Gubler~\cite{Gub13} a polyhedral subdivision of a vector space, with recession fan $\Sigma$ produces a toric scheme over a valuation ring $R$, having generic fiber $X(\Sigma)$ and special fiber isomorphic to a collection of toric varieties glued along boundary strata. For $\Sigma = \PP^1_{\trop}$ as in our case, this toric scheme may be constructed from the trivial family $\PP^1\times Spec(R)$, by toric blowup in the special fiber. Consequently, the special fiber comes with canonical contraction morphism to the main component. It is elementary to observe that the (marked) dual complex of this special fiber reproduces the polyhedral subdivision, where the vertex $0$ is the component in the special fiber. Loosely speaking, the tropical stable map is already aware of the expansion of the target. A closely related notion in algebraic geometry is the stable logarithmic map to an unexpanded target associated to a logarithmic stable map to an expansion, as studied by Abramovich, Marcus and Wise in~\cite{AMW12}. 
\end{remark}

\subsubsection{Rubber target} The role of a rubber target $\PP^1$, relative to $0$ and $\infty$ is played by a line graph, with two infinite edges. In other words, it is the datum of a subdivided target of the parametrized tropical map obtained after forgetting the location of $0$. 

We say that a graph $\Gamma_T$ is a $2$-marked line graph if it is a line graph with half edges attached to the each of the two $1$-valent vertices. That is, it is of the form depicted in Figure~\ref{fig: rubber-target}. We refer to the two half edges as  the $\pm \infty$ ends.

\begin{figure}[h!]
\begin{tikzpicture}
\draw [densely dotted] (-2,0)--(-1.5,0);
\draw (-1.5,0)--(-1,0);
\draw  (-1,0)--(1,0);
\draw  (,0)--(1,0);
\draw (1,0)--(2,0);
\draw [densely dotted] (3,0)--(2.5,0);
\draw (2,0)--(2.5,0);
\draw [ball color=black] (-1,0) circle (0.5mm);
\draw [ball color=black] (0,0) circle (0.5mm);
\draw [ball color=black] (1,0) circle (0.5mm);
\draw [ball color=black] (2,0) circle (0.5mm);
\end{tikzpicture}
\caption{A typical target for a tropical rubber stable map.}
\label{fig: rubber-target}
\end{figure}
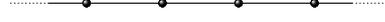

\begin{definition}\label{def-trop-rubber}
A tropical relative stable map to a rubber $\PP^1_{\trop}$ with ramification data $\bm x$ is an $n$-marked genus $0$ tropical curve $\Gamma$, a $2$-marked line graph $\Gamma_T$ and a morphism of metric graphs $g: \Gamma\to \Gamma_T$ subject to the following conditions. 
\begin{enumerate}[(R1)]
\item The finite edges and vertices of $\Gamma$ map to finite edges and vertices of $\Gamma_T$. 
\item The map $g$ is integer harmonic (not necessarily finite).  
\item If $x_i$ is positive (resp. negative), $g$ maps the marked end $p_i$ affinely to contain the end $+\infty$ (resp. $-\infty$) with expansion factor $x_i$. 
\item (Stability.) For every vertex $v$ of $\Gamma_T$, the preimage of $v$ contains a vertex of valence greater than $2$.
\end{enumerate}
\end{definition}

The data of the combinatorial types of source and target curves, with the morphism of graphs and with all expansion factors is referred to as a \textit{combinatorial type} of the relative stable map to rubber target.

\begin{remark}\label{rem-trop-parametrized}
Our formulation of Definition~\ref{def-trop-parametrized} is designed to make the connection to the literature clearer. At first glance, Definitions \ref{def-trop-parametrized} and \ref{def-trop-rubber} may seem unrelated. Their similarities are more apparent if we reformulate Definition \ref{def-trop-parametrized} as follows. 

\textit{
A tropical relative stable map to a parametrized $\PP^1_{\trop}$ with ramification data $\bm x$ is an $n$-marked genus $0$ tropical curve $\Gamma$, a $2$-marked line graph $\Gamma_T$ with a distinguished vertex $0$, and a morphism of metric graphs $g: \Gamma\to \Gamma_T$ such that
\begin{enumerate}[(P'1)]
\item The finite edges and vertices of $\Gamma$ map to finite edges and vertices of $\Gamma_T$. 
\item The map $g$ is finite, integer harmonic. 
\item If $x_i$ is positive (resp. negative), $g$ maps the marked end $p_i$ affinely to cover the end $-\infty$ (resp. $+\infty$) with expansion factor $x_i$. 
\item Stability: in the preimage of each vertex of $\Gamma_T$ except $0$, there is a more than 2-valent vertex.
\end{enumerate}
The equivalence to the data given in Definition \ref{def-trop-parametrized} follows after introducing $2$-valent vertices to $\PP^1_{\trop}$ at every image of a vertex (thus obtaining a line graph with a distinguished $0$), and introducing $2$-valent vertices to $\Gamma$ for every preimage of a vertex of the line graph.
}

\end{remark}

\subsection{Tropical moduli spaces and tautological morphisms}\label{subsec-tropicalmoduli}

Just as for classical relative stable maps, it is sometimes useful to have additional \textit{non-relative} marked ends, i.e.\ marked leaf edges which are  contracted by the map $g: \Gamma\to \Gamma_T$.

\begin{definition}
Let $\Theta = [\theta:\Gamma \to\Gamma_{T}]$ be a combinatorial type for a tropical relative stable map. An automorphism of $\Theta$ is a commuting square of automorphisms (of graphs) of source and target curves.
\[
\begin{tikzcd}
\Gamma \arrow{r}{\varphi_{\Gamma}} \arrow{d}[swap]{\theta} & \Gamma \arrow{d}{\theta}\\
\Gamma_{T} \arrow{r}[swap]{\varphi_{\Gamma_T}} & \Gamma_{T},
\end{tikzcd}
\]
where $\varphi_{\Gamma}$ and $\varphi_{\Gamma_T}$ are automorphisms preserving the expansion factors on each edge. 
\end{definition}

\begin{construction}\textnormal{
Standard constructions now allow us to build a tropical moduli space of relative stable maps, as in~\cite[Section 4]{ACP} and~\cite[Section 3.2]{CMR14}. This moduli space carries the structure of a polyhedral cone complex. Fixing a combinatorial type $\Theta = [\Gamma\to\Gamma_T]$, there is a polyhedral cone $\sigma_\Theta\cong \RR_{\geq 0}^B$ parametrizing metrizations of covers with combinatorial type $\Theta$, together with an identification of the underlying combinatorial type with $\Theta$. Here $B$ is the number of finite edges of the target $\Gamma_T$, plus the number of contracted edges of the source $\Gamma$. The (coarse) moduli space of tropical relative stable maps with combinatorial type $\Theta$ is identified with $\sigma_\Theta/Aut(\Theta)$. These cones glue together in the natural way to form a (coarse) moduli space of tropical relative stable maps. We denote the space of genus $0$ tropical relative stable maps with ramification data $\bm x$ and with $r$ additional markings, to parametrized $\PP^1_{\trop}$ by $ M_r^{\trop}(\PP^1_{\trop},\bm x)$. We denote the corresponding tropical rubber space $ M_r^{\trop}(\bm x)$.  By allowing the edge lengths of internal edges to become infinite, we obtain the canonical compactifications of these cone complexes, denoted by $\overline M_r^{\trop}(\PP^1_{\trop},\bm x)$ and $\overline M^{\trop}_r(\bm x)$. We drop the subscript $r$, when $r = 0$.}
\end{construction}

Just as in the classical setting, there are tropical tautological morphisms. These are morphisms of polyhedral complexes with integral structure. We denote these tropical tautological morphisms using fraktur letters. When the target is parametrized, there exist evaluation morphisms 
\[
\mathfrak{ev}_i: \overline M_r^{\trop}(\PP^1_{\trop},\bm x) \to \PP^1_{\trop},
\]
sending the a map $[\Gamma\to \PP^1_{\trop}]$ to the image of the $i$th marked end in $\PP^1_{\trop}$. 

For the rubber space, we also have a stabilization morphism

\[
\mathfrak{st}: \overline M_r^{\trop}(\bm x) \to \overline{M}_{0,n+r}^{\trop},
\]

sending a map $g: \Gamma\to \Gamma_T$ to the stabilization of $\Gamma$. Recall that on the tropical side, stabilization involves forgetting $2$-valent vertices and merging the edges they are incident to. Here, $\overline{M}_{0,n+r}^{\trop}$ is the moduli space of $(n+r)$-marked rational stable tropical curves, sometimes referred to as the space of phylogenetic trees~\cite{BHV01, SS04a, GKM07, GM07}.

Furthermore, there is a branch morphism

\[
\mathfrak{br}: \overline M_r^{\trop}(\bm x) \to [\overline M^{\trop}_{0,2+r\cdot \epsilon}/S_r]
\]
sending a map $g: \Gamma\to \Gamma_T$ to the line graph $\Gamma_T$. Here, $\overline M^{\trop}_{0,2+r\cdot \epsilon}$ denotes the tropical Losev-Manin space parametrizing line graphs with marked vertices, and $[\overline M^{\trop}_{0,2+r\cdot \epsilon}/S_r]$ its quotient under the action of the symmetric group permuting the vertex markings~\cite{CHMR14}.

Similar tropical moduli spaces and morphisms have been constructed in the literature before, e.g.\ in~\cite{GKM07, MR08, Mi07}. The following proposition (c.f. Proposition \ref{prop-stab}) follows from~\cite[Proposition 4.7]{GKM07}.

\begin{proposition}\label{prop-troprubber}
The stabilization morphism $\mathfrak{st}:\overline M^{\trop}_r(\bm x)\to \overline M^{\trop}_{0,n+r}$ is a bijection of the sets underlying the respective cone complexes.
\end{proposition}

\subsubsection{Fan structures on tropical moduli spaces} In work of Gathmann, Kerber, and the second author, the above fact is used to endow the set $M^{\trop}_r(\bm x)$ with the structure of a simplicial balanced fan. This is done by lifting the fan structure on the moduli space $M^{\trop}_{0,n+r}$ via the above bijection. This fan structure is described as follows. The set $\overline M^{\trop}_{0,n+r}$ can be embedded as a simplicial balanced fan that we denote $\Delta_{n+r}$; via the so-called \textit{distance map}. This map sends a tropical curve to the vector of distances of its leaves \cite{GKM07, Mi07, SS04a}.\label{deltan} By abuse of notation, we identify a tropical curve with its image under this embedding in the following.
The rays of the fan $\Delta_n$ are given by tropical curves with only one interior edge and two vertices, and marked ends labeled by a subset $I\subset \{1,\ldots,n\}$ adjacent to one vertex while the ends labeled by $I^c$ are adjacent to the other vertex. Here, the size of both $I$ and $I^c$ is required to be at least two to satisfy the stability condition. Following standard notation, we denote the ray corresponding to the subset $I$ by $v_I$. Note that $v_I=v_{I^c}$ in this notation.

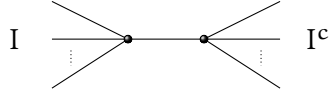
\begin{figure}[h!]
\begin{tikzpicture}
\draw (-1,0.5)--(0,0);
\draw (-1,0)--(0,0);
\draw (-1,-0.65)--(0,0);
\draw (0,0)--(1,0);
\draw (1,0)--(2,0.5);
\draw (1,0)--(2,0);
\draw (1,0)--(2,-0.65);

\draw [densely dotted] (-0.75,-0.15)--(-0.75,-0.35);
\draw [densely dotted] (1.75,-0.15)--(1.75,-0.35);

\draw [ball color=black] (0,0) circle (0.5mm);
\draw [ball color=black] (1,0) circle (0.5mm);

\node at (-1.5,0) {$I$};
\node at (2.5,0) {$I^c$};
\end{tikzpicture}
\caption{The combinatorial type associated to the ray $v_I = v_{I^c}$ of $M_{0,n}^{t\!r\!o\!p}$.}
\label{fig: rubber-target}
\end{figure}

Having a tropical moduli space that is a balanced fan allows the authors of~\cite{GKM07} to apply tropical intersection theoretic techniques to compute intersection numbers on $M^{\trop}_{0,n}$. For such applications, more refined fan structures on the set $M^{\trop}_{0,n}$ given by subdivisions are not important. However, in the present work, we observe that the fan structure given by the cones corresponding to combinatorial types of relative stable maps is essential to obtain deeper insights into the connection of the tropical spaces with their classical counterparts. 

From a toric perspective, a refinement $\Delta'\to \Delta$ yields a proper birational morphism $\varphi: X(\Delta')\to X(\Delta)$, which restricts to the identity on the dense torus. In other words, the refined fan structures are important only in situations where we work with a compactified moduli space. This is a contrast to the situations considered in~\cite{GKM07, Gro14}.

Let $\Sigma$ be a cone complex and suppose $\sigma = \langle v_1,\ldots, v_k\rangle$ is a simplicial cone. Let $v_\star = v_1+\cdots+v_k$. Let $\Sigma'(\sigma)$ be the collection of cones generated by subsets of $\{v_\star,v_1,\ldots, v_k\}$ that do not contain $\{v_1,\ldots, v_n\}$. The \textit{stellar subdivision} of $\Sigma$ along $\sigma$ is defined as
\[
\Sigma^\star(\sigma) = (\Sigma\setminus \{\sigma\})\cup \Sigma'(\sigma).
\]
This process introduces a new ray generated by $v_\star$. Analogously, given a vector of positive integer weights $\bm w = (w_1,\ldots, w_k)$, the $\bm w$-weighted stellar subdivision of $\sigma$ and $\Sigma$ is defined by introducing the new ray $v_{\bm w} = \sum_{i=1}^k w_iv_i$ in place of $v_\star$ throughout. 

The following proposition is a key ingredient for the proof of Theorem \ref{thm: tropical-compactification}.

\begin{proposition}\label{prop-subdivision}
The stabilization morphism $\mathfrak{st}:\overline M^{\trop}_r(\bm x)\to \overline M^{\trop}_{0,n+r}$ induces an iterated weighted stellar subdivision of simplicial fans. We denote the subdivided fan by $\Delta_{\bm x,r}^{\rub}$ and consider the stabilization morphism as $$\mathfrak{st}:\Delta_{\bm x,r}^{\rub}\to \Delta_{n+r}.$$
\end{proposition}
Following our standard convention, we leave out the subscript $r$ if there are no additional marked points.

An analogous statement holds for maps to a parametrized target. In order to rigidify the problem, we assume that there is at least one non-relative marking, i.e.\ $r>0$.
As in Proposition \ref{prop-troprubber},~\cite[Proposition 4.7]{GKM07} tells us that the stabilization morphism times the evaluation of the marked end $n+1$ (i.e.\ the first non-relative point) $$\mathfrak{st}\times \mathfrak{ev}_{n+1}:\overline M^{\trop}_r(\PP^1_{\trop},\bm x)\to \overline M^{\trop}_{0,n+r}\times\PP^1_{\trop}$$ is a bijection of the underlying sets. The following proof is easily adapted to this situation. We denote the subdivided fan by $\Delta_{\bm x,r}^{\parm}$ and consider the morphism as $$\mathfrak{st}\times \mathfrak{ev}_{n+1}:\Delta_{\bm x,r}^{\parm}\to \Delta_{n+r}\times\PP^1_{\trop}.$$

\begin{proof}
Recall that a combinatorial type of a rubber map $\Theta = [\Gamma\to \Gamma_T]$ is given by the data of the combinatorial types of the source and target, the underlying morphism of finite graphs, and all expansion factors. It is a consequence of Definition~\ref{def-trop-parametrized} that a morphism of graphs to a line graph is equivalent to the data of an ordering on the vertices of $\Gamma$, induced by the left-to-right ordering of their images in $\Gamma_T$. Furthermore, since $\Gamma$ has genus $0$, the expansion factors are determined by the ramification profile $\bm x$~\cite[Definition 2]{BCM}. In other words, the combinatorial type is determined by a combinatorial type for the source $\Gamma$ together with an ordering on the vertices that is compatible with the ramification data. 

Let $\Theta$ be a combinatorial type, such that the cone $\sigma_\Theta$ has maximal dimension, i.e. with the maximal number of bounded edges in $\Gamma_T$. For simplicity, we first deal with the case where all expansion factors are nonzero. Observe that $\Theta$ can be maximal dimensional if and only if the stabilization of the source curve is trivalent and if the images of the trivalent vertices are distinct in $\Gamma_T$. In this case, the combinatorial type $\Theta$ is determined by the stabilization of the source curve and a total ordering on the trivalent vertices. Passing to the stabilization, we lose the data of the total ordering on these vertices. 

Conversely, given a maximal dimensional cone $\sigma$ in $M^{\trop}_{0,n+r}$, we subdivide $\sigma$ into cones corresponding to all possible total orderings of the trivalent vertices that respect the ramification profile. The rays introduced in this subdivision are determined by taking a possible total ordering and changing all but one of the inequalities to equalities. If such a ray is a face of the fan $\overline M^{\trop}_{0,n+r}$ then it is a positive integer combination of such rays, where the coefficients are imposed by the ramification data. We introduce these new rays and perform (weighted) stellar subdivision to obtain the collection of cones of $M^{\trop}_r(\bm x)$. 

The case where a combinatorial type $\Theta$ of tropical relative stable maps has zero expansion factors is similar. First, note that given $\Theta$, we obtain a ray $v_I$ in $M^{\trop}_{0,n+r}$, where $I\cup I^c = \{1,\ldots, n\}$ determines a partition of the markings obtained by cutting this contracted internal edge. This ray is a $1$-dimensional face of any maximal cone of $M_r^{\trop}(\bm x)$ refining the cone of $M^{\trop}_{0,n+r}$ obtained after stabilization of $\Theta$. If an edge $e$ has expansion factor $0$, then its endpoints are mapped to the same vertex of $\Gamma_T$. Consequently, we can repeat the argument above for the remaining rays and inequalities imposed by the images of the remaining vertices.
\end{proof}

\begin{example}
 Let $n=6$, $\bm x=(-4,-4,5,1,1,1)$ and $r=0$. Consider the maximal cone of $\overline M^{\trop}_{0,6}$ spanned by the rays $v_{\{1,4\}}$, $v_{\{1,3,4\}}$ and $v_{\{5,6\}}$. The corresponding combinatorial type $\alpha$ of $6$-marked tropical curves is shown in Figure~\ref{oi}. On the right, we see a combinatorial type $\alpha'$ of a tropical relative stable map with ramification data $\bm x$ whose stabilization is $\alpha$. 
 
 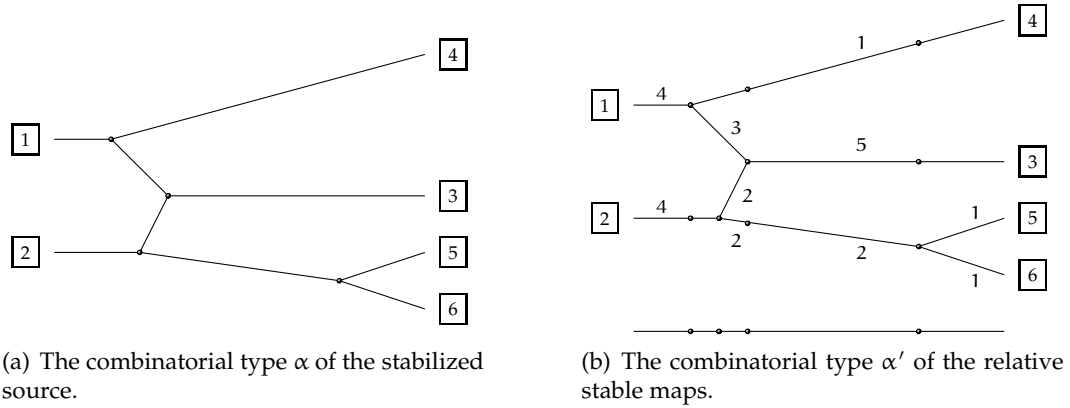
\begin{figure}[h!]
\hfill
\subfigure[The combinatorial type $\alpha$ of the stabilized source.]{\begin{tikzpicture}[scale=1.5]
\coordinate (0) at (-2,1);
\coordinate (1) at (-2,2);
\coordinate (2) at (-1.5,2);
\coordinate (3) at (-1.25,1);
\coordinate (4) at (-1,1.5);
\coordinate (5) at (1.25,2.75);
\coordinate (6) at (1.25,1.5);
\coordinate (7) at (0.5, 0.75);
\coordinate (8) at (1.25,1);
\coordinate (9) at (1.25,0.5);

\coordinate (A) at (-2,0);
\coordinate (B) at (-1.5,0);
\coordinate (C) at (-1.25,0);
\coordinate (D) at (-1,0);
\coordinate (E) at (.5,0);
\coordinate (F) at (1.25,0);

\draw (1)--(2)--(4)--(3)--(0);
\draw (2)--(5);
\draw (4)--(6);
\draw (3)--(7);
\draw (9)--(7)--(8);

\draw[ball color=black] (2) circle (0.2mm);
\draw[ball color=black] (3) circle (0.2mm);
\draw[ball color=black] (4) circle (0.2mm);
\draw[ball color=black] (7) circle (0.2mm);

\node at (-2.25,1) {\tiny \boxed{$2$}};
\node at (-2.25,2) {\tiny \boxed{$1$}};
\node at (1.5,2.75) {\tiny \boxed{$4$}};
\node at (1.5,1.5) {\tiny \boxed{$3$}};
\node at (1.5,1) {\tiny \boxed{$5$}};
\node at (1.5,0.5) {\tiny \boxed{$6$}};
\end{tikzpicture}
}
\hfill
\subfigure[The combinatorial type $\alpha'$ of the relative stable maps.]{\begin{tikzpicture}
[scale=1.5]
\coordinate (0) at (-2,1);
\coordinate (1) at (-2,2);
\coordinate (2) at (-1.5,2);
\coordinate (3) at (-1.25,1);
\coordinate (4) at (-1,1.5);
\coordinate (5) at (1.25,2.75);
\coordinate (6) at (1.25,1.5);
\coordinate (7) at (0.5, 0.75);
\coordinate (8) at (1.25,1);
\coordinate (9) at (1.25,0.5);

\coordinate (A) at (-2,0);
\coordinate (B) at (-1.5,0);
\coordinate (C) at (-1.25,0);
\coordinate (D) at (-1,0);
\coordinate (E) at (.5,0);
\coordinate (F) at (1.25,0);

\draw (1)--(2)--(4)--(3)--(0);
\draw (2)--(5);
\draw (4)--(6);
\draw (3)--(7);
\draw (9)--(7)--(8);
\draw (A)--(B)--(C)--(D)--(E)--(F);

\draw[ball color=black] (2) circle (0.2mm);
\draw[ball color=black] (3) circle (0.2mm);
\draw[ball color=black] (4) circle (0.2mm);
\draw[ball color=black] (7) circle (0.2mm);
\draw[ball color=black] (-1.5,1) circle (0.2mm);
\draw[ball color=black] (-1,2.14) circle (0.2mm);
\draw[ball color=black] (-1,0.955) circle (0.2mm);
\draw[ball color=black] (0.5,1.5) circle (0.2mm);
\draw[ball color=black] (0.5,2.55) circle (0.2mm);

\draw[ball color=black] (B) circle (0.2mm);
\draw[ball color=black] (C) circle (0.2mm);
\draw[ball color=black] (D) circle (0.2mm);
\draw[ball color=black] (E) circle (0.2mm);

\node at (-2.25,1) {\tiny \boxed{$2$}};
\node at (-2.25,2) {\tiny \boxed{$1$}};
\node at (1.5,2.75) {\tiny \boxed{$4$}};
\node at (1.5,1.5) {\tiny \boxed{$3$}};
\node at (1.5,1) {\tiny \boxed{$5$}};
\node at (1.5,0.5) {\tiny \boxed{$6$}};

\node at (-1.75,1.1) {\tiny $4$};
\node at (-1.75,2.1) {\tiny $4$};
\node at (-1,1.2) {\tiny $2$};
\node at (-1.1,0.8) {\tiny $2$};
\node at (-1.1,1.8) {\tiny $3$};
\node at (0,2.55) {\tiny $1$};
\node at (0,1.65) {\tiny $5$};
\node at (0,.7) {\tiny $2$};
\node at (1,.45) {\tiny $1$};
\node at (1,1.05) {\tiny $1$};
\end{tikzpicture}

}
\hfill
\caption{Combinatorial types of maximal cells in  $M^{\trop}_{0,6}$ and $M^{\trop}(\bm x)$ for $\bm x = (-4,-4,5,1,1,1)$.}\label{oi}
\end{figure}

The type $\alpha'$ is determined by $\alpha$, together with an ordering of the $3$-valent vertices. We can name the $3$-valent vertices by the marked end of $\{2,3,4,5\}$ adjacent to them in the stabilization. Using this notation, the ordering of $\alpha'$ equals $4<2<3<5$. We denote this ordering by $O_1$.

The ramification data imposes the relations $4\leq 3$, $2\leq 3$ and $2\leq 5$ for the combinatorial type $\alpha$. Thus, the following orderings are also allowed: $$O_2: 2<4<3<5, \ \ \ \ \ \ O_3: 4<2<5<3, \ \ \ \ \ \ O_4: 2<4<5<3, \ \ \ \ \ \ O_5: 2<5<4<3.$$ Figure~\ref{fig: m06-subdivision} shows the combinatorics of the subdivided cone of $\overline M^{\trop}_{0,6}$.

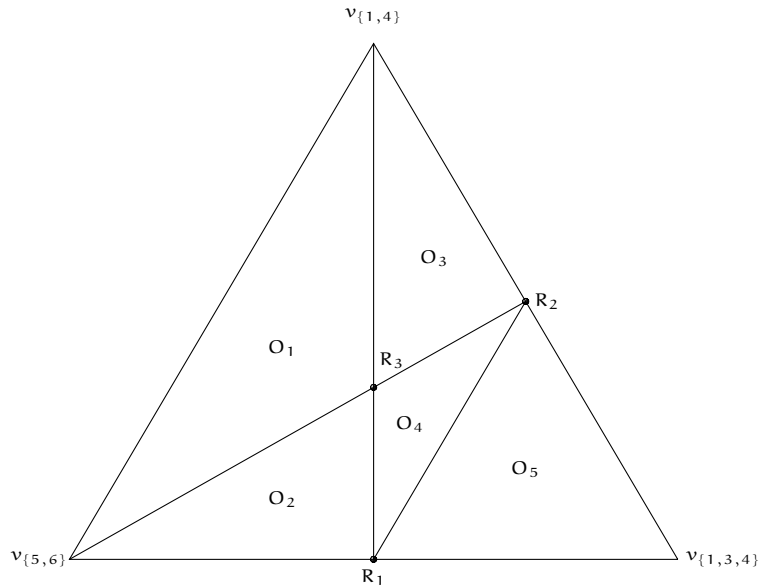
\begin{figure}[h!]
\begin{tikzpicture}[scale=0.8]
\coordinate (1) at (4,0);
\coordinate (2) at (-6,0);
\coordinate (3) at (-1, 8.55);
\coordinate (R1) at (-1,0);
\coordinate (R2) at (1.5,4.275);
\coordinate (R3) at (-1,2.85);

\draw[ball color=black] (R1) circle (0.5mm);
\draw[ball color=black] (R2) circle (0.5mm);
\draw[ball color=black] (R3) circle (0.5mm);
\draw (1)--(2)--(3)--(1);
\draw (2)--(R2)--(R1)--(3);

\node at (4.75,0) {\tiny$v_{\{1,3,4\}}$};
\node at (-6.5,0) {\tiny$v_{\{5,6\}}$};
\node at (-1,9) {\tiny$v_{\{1,4\}}$};

\node at (0,5) {\tiny$O_3$};
\node at (-2.5,3.5) {\tiny$O_1$};
\node at (-2.5,1) {\tiny$O_2$};
\node at (-.4,2.25) {\tiny$O_4$};
\node at (1.5,1.5) {\tiny $O_5$};

\node at (-1,-0.3) {\tiny $R_1$};
\node at (1.85,4.275) {\tiny $R_2$}; 
\node at (-0.7,3.3) {\tiny $R_3$};

\end{tikzpicture}
\caption{The subdivision induced by the stabilization morphism.}
\label{fig: m06-subdivision}
\end{figure}

To illustrate this process, we compute the additional rays prescribed by the ordering $O_2$. Figure~\ref{fig: o2-types} shows the combinatorial type corresponding to this order, plus the 3 combinatorial types of maps where we let two of the inequalities become equalities and leave only one strict inequality. The stabilization of the map given by $2=4=3<5$ is the ray $v_{\{5,6\}}$. The stabilization of the map given by $2<4=3=5$ is 
\[
R_1=v_{\{1,3,4\}}+v_{\{5,6\}}.
\] 
Similarly, the ray for the map given by $2=4<3=5$ is \[
R_2=2\cdot v_{\{1,4\}}+3\cdot v_{\{1,3,4\}}+3\cdot v_{\{5,6\}}. 
\]
The linear coefficients here are determined by the expansion factors of the interior edges.
In addition to those two rays, we need to also add  $R_3=2\cdot v_{\{1,4\}}+3\cdot v_{\{1,3,4\}}$. This can  be seen by computing the boundary of the open cone determined by the ordering $O_3$. We could for example first add $R_1$, then $R_3  $ and finally $R_2$.

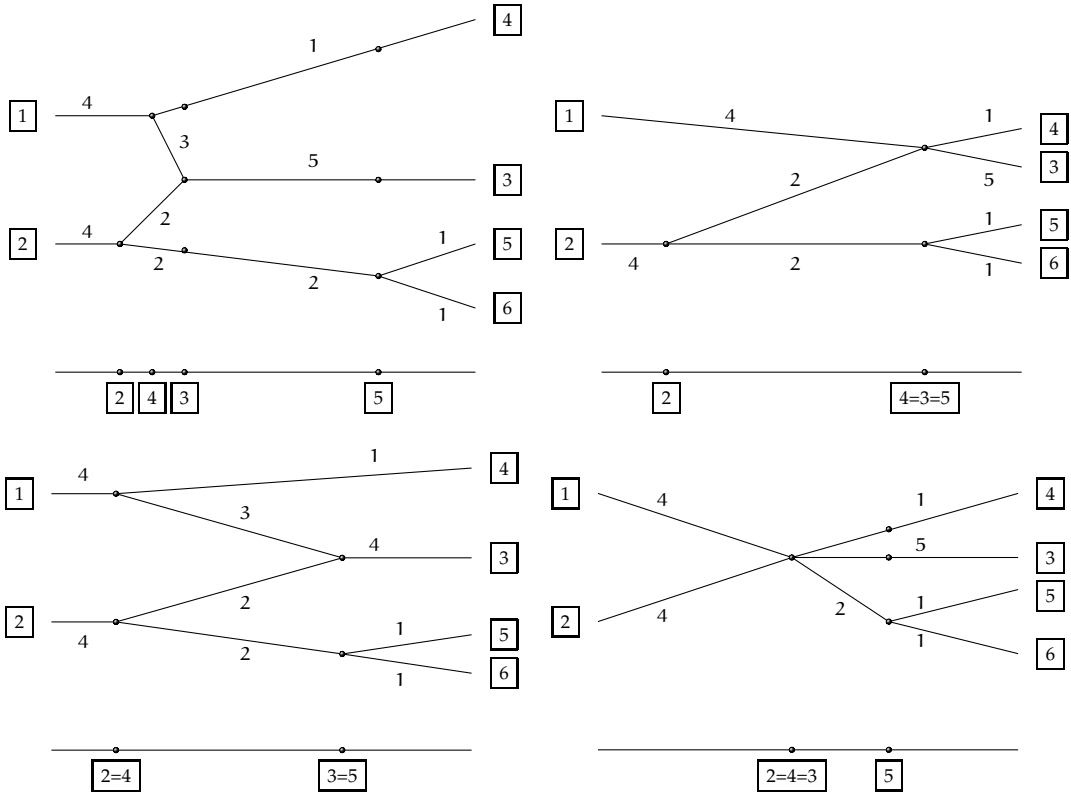
\begin{figure}[h!]
\subfigure{
\begin{tikzpicture}
[scale=1.7]
\coordinate (0) at (-2,1);
\coordinate (1) at (-2,2);
\coordinate (2) at (-1.25,2);
\coordinate (3) at (-1.5,1);
\coordinate (4) at (-1,1.5);
\coordinate (5) at (1.25,2.75);
\coordinate (6) at (1.25,1.5);
\coordinate (7) at (0.5, 0.75);
\coordinate (8) at (1.25,1);
\coordinate (9) at (1.25,0.5);

\coordinate (A) at (-2,0);
\coordinate (B) at (-1.5,0);
\coordinate (C) at (-1.25,0);
\coordinate (D) at (-1,0);
\coordinate (E) at (.5,0);
\coordinate (F) at (1.25,0);

\node at (-1.5,-0.2) {\tiny \boxed{$2$}};
\node at (-1.25,-0.2) {\tiny \boxed{$4$}};
\node at (-1,-0.2) {\tiny \boxed{$3$}};
\node at (.5,-0.2) {\tiny \boxed{$5$}};

\draw (1)--(2)--(4)--(3)--(0);
\draw (2)--(5);
\draw (4)--(6);
\draw (3)--(7);
\draw (9)--(7)--(8);
\draw (A)--(B)--(C)--(D)--(E)--(F);

\draw[ball color=black] (2) circle (0.2mm);
\draw[ball color=black] (3) circle (0.2mm);
\draw[ball color=black] (4) circle (0.2mm);
\draw[ball color=black] (7) circle (0.2mm);
\draw[ball color=black] (-1.5,1) circle (0.2mm);
\draw[ball color=black] (-1,2.07) circle (0.2mm);
\draw[ball color=black] (-1,0.95) circle (0.2mm);
\draw[ball color=black] (0.5,1.5) circle (0.2mm);
\draw[ball color=black] (0.5,2.52) circle (0.2mm);

\draw[ball color=black] (B) circle (0.2mm);
\draw[ball color=black] (C) circle (0.2mm);
\draw[ball color=black] (D) circle (0.2mm);
\draw[ball color=black] (E) circle (0.2mm);

\node at (-1.75,1.1) {\tiny $4$};
\node at (-1.75,2.1) {\tiny $4$};
\node at (-1.15,1.2) {\tiny $2$};
\node at (-1.2,0.85) {\tiny $2$};
\node at (-1,1.8) {\tiny $3$};
\node at (0,2.55) {\tiny $1$};
\node at (0,1.65) {\tiny $5$};
\node at (0,.7) {\tiny $2$};
\node at (1,.45) {\tiny $1$};
\node at (1,1.05) {\tiny $1$};

\node at (-2.25,1) {\tiny \boxed{$2$}};
\node at (-2.25,2) {\tiny \boxed{$1$}};
\node at (1.5,2.75) {\tiny \boxed{$4$}};
\node at (1.5,1.5) {\tiny \boxed{$3$}};
\node at (1.5,1) {\tiny \boxed{$5$}};
\node at (1.5,0.5) {\tiny \boxed{$6$}};
\end{tikzpicture}
}
\subfigure{\begin{tikzpicture}
[scale=1.7]
\coordinate (0) at (-2,1);
\coordinate (1) at (-2,2);
\coordinate (2) at (-1.5,1);
\coordinate (3) at (0.5,1.75);
\coordinate (4) at (0.5,1);
\coordinate (5) at (1.25,1.9);
\coordinate (6) at (1.25,1.6);
\coordinate (7) at (1.25,1.15);
\coordinate (8) at (1.25,0.85);

\coordinate (A) at (-2,0);
\coordinate (B) at (-1.5,0);
\coordinate (C) at (-1.25,0);
\coordinate (D) at (-1,0);
\coordinate (E) at (.5,0);
\coordinate (F) at (1.25,0);

\node at (-1.5,-0.2) {\tiny\boxed{$2$}};
\node at (.5,-0.2) {\tiny\boxed{$4=3=5$}};

\draw (0)--(2)--(3)--(1);
\draw (2)--(4);
\draw (5)--(3)--(6);
\draw (8)--(4)--(7);
\draw (A)--(B)--(C)--(D)--(E)--(F);

\draw[ball color=black] (2) circle (0.2mm);
\draw[ball color=black] (3) circle (0.2mm);
\draw[ball color=black] (4) circle (0.2mm);

\draw[ball color=black] (B) circle (0.2mm);
\draw[ball color=black] (E) circle (0.2mm);

\node at (-2.25,1) {\tiny \boxed{$2$}};
\node at (-2.25,2) {\tiny \boxed{$1$}};
\node at (1.5,1.9) {\tiny \boxed{$4$}};
\node at (1.5,1.6) {\tiny \boxed{$3$}};
\node at (1.5,1.15) {\tiny \boxed{$5$}};
\node at (1.5,0.85) {\tiny \boxed{$6$}};

\node at (-1,2) {\tiny $4$};
\node at (-1.75,0.85) {\tiny $4$};
\node at (-0.5,1.5) {\tiny $2$};
\node at (-0.5, 0.85) {\tiny $2$};
\node at (1, 2) {\tiny $1$};
\node at (1, 1.5) {\tiny $5$};
\node at (1,1.2) {\tiny $1$};
\node at (1,0.8) {\tiny $1$};
\end{tikzpicture}
}
\subfigure{\begin{tikzpicture}
[scale=1.7]
\coordinate (0) at (-2,1);
\coordinate (1) at (-2,2);
\coordinate (2) at (-1.5,1);
\coordinate (3) at (-1.5,2);
\coordinate (4) at (0.25,1.5);
\coordinate (5) at (0.25,0.75);
\coordinate (6) at (1.25,2.2);
\coordinate (7) at (1.25,1.5);
\coordinate (8) at (1.25,0.9);
\coordinate (9) at (1.25,0.6);

\coordinate (A) at (-2,0);
\coordinate (B) at (-1.5,0);
\coordinate (C) at (-1.25,0);
\coordinate (D) at (-1,0);
\coordinate (E) at (.25,0);
\coordinate (F) at (1.25,0);

\draw (1)--(3)--(4)--(2)--(0);
\draw (2)--(5);
\draw (3)--(6);
\draw (4)--(7);
\draw (8)--(5)--(9);

\draw (A)--(B)--(C)--(D)--(E)--(F);

\draw[ball color=black] (2) circle (0.2mm);
\draw[ball color=black] (3) circle (0.2mm);
\draw[ball color=black] (4) circle (0.2mm);
\draw[ball color=black] (5) circle (0.2mm);

\draw[ball color=black] (B) circle (0.2mm);
\draw[ball color=black] (E) circle (0.2mm);

\node at (-2.25,1) {\tiny \boxed{$2$}};
\node at (-2.25,2) {\tiny \boxed{$1$}};
\node at (1.5,2.2) {\tiny \boxed{$4$}};
\node at (1.5,1.5) {\tiny \boxed{$3$}};
\node at (1.5,.9) {\tiny \boxed{$5$}};
\node at (1.5,0.6) {\tiny \boxed{$6$}};
\node at (-1.5,-0.2) {\tiny \boxed{$2=4$}};
\node at (.25,-0.2) {\tiny \boxed{$3=5$}};

\node at (-1.75,0.85) {\tiny $4$};
\node at (-1.75,2.15) {\tiny $4$};
\node at (.5,2.3) {\tiny $1$};
\node at (.5,1.6) {\tiny $4$};
\node at (.7,.95) {\tiny $1$};
\node at (.7,.55) {\tiny $1$};
\node at (-0.5,0.75) {\tiny $2$};
\node at (-0.5,1.15) {\tiny $2$};
\node at (-0.5,1.85) {\tiny $3$};

\end{tikzpicture}
}
\subfigure{\begin{tikzpicture}
[scale=1.7]
\coordinate (0) at (-2,1);
\coordinate (1) at (-2,2);
\coordinate (2) at (-0.5,1.5);
\coordinate (3) at (0.25,1);
\coordinate (4) at (0.25,1.5);
\coordinate (5) at (0.25,1.72);
\coordinate (6) at (1.25,1.25);
\coordinate (7) at (1.25,0.75);
\coordinate (8) at (1.25, 1.5);
\coordinate (9) at (1.25,2);

\coordinate (A) at (-2,0);
\coordinate (B) at (-1.5,0);
\coordinate (C) at (-1.25,0);
\coordinate (D) at (-0.5,0);
\coordinate (E) at (0.25,0);
\coordinate (F) at (1.25,0);

\draw (1)--(2);
\draw (4)--(2)--(3);
\draw (0)--(2);
\draw (4)--(8);
\draw (7)--(3)--(6);

\draw (A)--(B)--(C)--(D)--(E)--(F);

\draw[ball color=black] (2) circle (0.2mm);
\draw[ball color=black] (3) circle (0.2mm);
\draw[ball color=black] (4) circle (0.2mm);
\draw[ball color=black] (5) circle (0.2mm);

\draw (2)--(9);

\draw[ball color=black] (D) circle (0.2mm);
\draw[ball color=black] (E) circle (0.2mm);

\node at (-2.25,1) {\tiny \boxed{$2$}};
\node at (-2.25,2) {\tiny \boxed{$1$}};
\node at (1.5,2) {\tiny \boxed{$4$}};
\node at (1.5,1.5) {\tiny \boxed{$3$}};
\node at (1.5,1.2) {\tiny \boxed{$5$}};
\node at (1.5,0.75) {\tiny \boxed{$6$}};

\node at (-1.5,1.95) {\tiny $4$};
\node at (-1.5,1.05) {\tiny $4$};
\node at (-0.125,1.1) {\tiny $2$};
\node at (.5,1.6) {\tiny $5$};
\node at (.5,1.95) {\tiny $1$};
\node at (.5,1.15) {\tiny $1$};
\node at (.5,0.85) {\tiny $1$};

\node at (-0.5,-0.2) {\tiny \boxed{$2=4=3$}};
\node at (0.25,-0.2) {\tiny \boxed{$5$}};
\end{tikzpicture}
}
\caption{The combinatorial type $O_2$ and the combinatorial types of its three $1$-dimensional faces. Boxed numbers on the target graph indicate a labeling of the essential vertices on the source. These vertices are partially ordered by their image in the target.}
\label{fig: o2-types}
\end{figure}

Note that the neighbouring cones are also subdivided. Figure~\ref{fig: neighboring-cone} shows exemplarily the cone of  $\overline M^{\trop}_{0,6}$ spanned by the rays $v_{\{3,4\}}$, $v_{\{1,3,4\}}$ and $v_{\{5,6\}}$ and its subdivision, indicating the combinatorial types of maps corresponding to the cones in the subdivision.

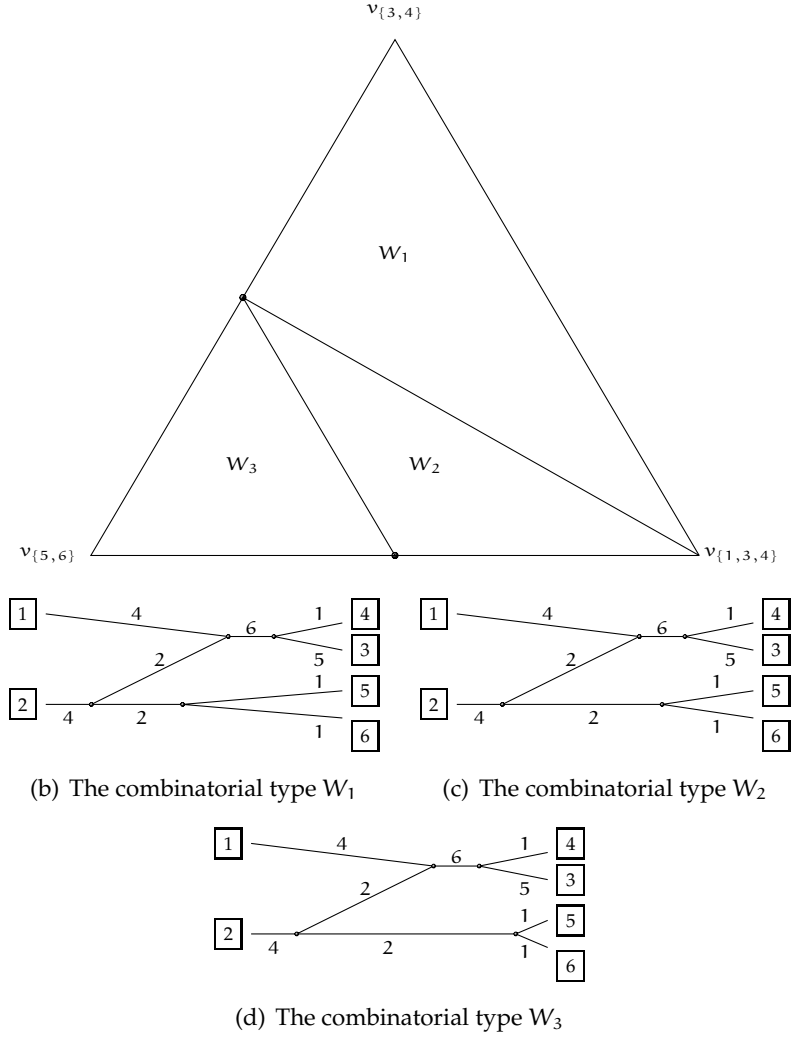
\begin{figure}[h!]
\subfigure{\begin{tikzpicture}[scale=0.8]
\coordinate (1) at (4,0);
\coordinate (2) at (-6,0);
\coordinate (3) at (-1, 8.55);
\coordinate (R1) at (-1,0);
\coordinate (R3) at (-3.5,4.275);

\draw[ball color=black] (R1) circle (0.5mm);
\draw[ball color=black] (R3) circle (0.5mm);
\draw (1)--(2)--(3)--(1);
\draw (R1)--(R3)--(1);

\node at (4.7,0) {\tiny $v_{\{1,3,4\}}$};
\node at (-6.7,0) {\tiny $v_{\{5,6\}}$};
\node at (-1,9) {\tiny $v_{\{3,4\}}$};

\node at (-1,5) {\tiny $W_1$};
\node at (-0.5,1.5) {\tiny $W_2$};
\node at (-3.5,1.5) {\tiny $W_3$};

\end{tikzpicture}}
\subfigure[The combinatorial type $W_1$]{
\begin{tikzpicture}
[scale=1.2]
\coordinate (0) at (-2,1);
\coordinate (1) at (-2,2);
\coordinate (2) at (-1.5,1);
\coordinate (3) at (0.5,1.75);
\coordinate (3') at (0,1.75);
\coordinate (4) at (-0.5,1);
\coordinate (5) at (1.25,1.9);
\coordinate (6) at (1.25,1.6);
\coordinate (7) at (1.25,1.15);
\coordinate (8) at (1.25,0.85);

\draw (0)--(2)--(3')--(1);
\draw (3)--(3');
\draw (2)--(4);
\draw (5)--(3)--(6);
\draw (8)--(4)--(7);

\draw[ball color=black] (2) circle (0.2mm);
\draw[ball color=black] (3) circle (0.2mm);
\draw[ball color=black] (3') circle (0.2mm);
\draw[ball color=black] (4) circle (0.2mm);

\node at (-2.25,1) {\tiny \boxed{$2$}};
\node at (-2.25,2) {\tiny \boxed{$1$}};
\node at (1.5,2) {\tiny \boxed{$4$}};
\node at (1.5,1.6) {\tiny \boxed{$3$}};
\node at (1.5,1.15) {\tiny \boxed{$5$}};
\node at (1.5,0.65) {\tiny \boxed{$6$}};

\node at (-1,2) {\tiny $4$};
\node at (-1.75,0.85) {\tiny $4$};
\node at (-0.75,1.5) {\tiny $2$};
\node at (-0.95, 0.85) {\tiny $2$};
\node at (1, 2) {\tiny $1$};
\node at (1, 1.5) {\tiny $5$};
\node at (1,1.25) {\tiny $1$};
\node at (1,0.7) {\tiny $1$};
\node at (0.25,1.85) {\tiny $6$};
\end{tikzpicture}
}
\subfigure[The combinatorial type $W_2$]{
\begin{tikzpicture}
[scale=1.2]
\coordinate (0) at (-2,1);
\coordinate (1) at (-2,2);
\coordinate (2) at (-1.5,1);
\coordinate (3) at (0.5,1.75);
\coordinate (3') at (0,1.75);
\coordinate (4) at (0.25,1);
\coordinate (5) at (1.25,1.9);
\coordinate (6) at (1.25,1.6);
\coordinate (7) at (1.25,1.15);
\coordinate (8) at (1.25,0.85);

\draw (0)--(2)--(3')--(1);
\draw (3)--(3');
\draw (2)--(4);
\draw (5)--(3)--(6);
\draw (8)--(4)--(7);

\draw[ball color=black] (2) circle (0.2mm);
\draw[ball color=black] (3) circle (0.2mm);
\draw[ball color=black] (3') circle (0.2mm);
\draw[ball color=black] (4) circle (0.2mm);

\node at (-2.25,1) {\tiny \boxed{$2$}};
\node at (-2.25,2) {\tiny \boxed{$1$}};
\node at (1.5,2) {\tiny \boxed{$4$}};
\node at (1.5,1.6) {\tiny \boxed{$3$}};
\node at (1.5,1.15) {\tiny \boxed{$5$}};
\node at (1.5,0.65) {\tiny \boxed{$6$}};

\node at (-1,2) {\tiny $4$};
\node at (-1.75,0.85) {\tiny $4$};
\node at (-0.75,1.5) {\tiny $2$};
\node at (-0.5, 0.85) {\tiny $2$};
\node at (1, 2) {\tiny $1$};
\node at (1, 1.5) {\tiny $5$};
\node at (0.85,1.25) {\tiny $1$};
\node at (0.85,0.75) {\tiny $1$};
\node at (0.25,1.85) {\tiny $6$};
\end{tikzpicture}
}
\subfigure[The combinatorial type $W_3$]{
\begin{tikzpicture}
[scale=1.2]
\coordinate (0) at (-2,1);
\coordinate (1) at (-2,2);
\coordinate (2) at (-1.5,1);
\coordinate (3) at (0.5,1.75);
\coordinate (3') at (0,1.75);
\coordinate (4) at (0.9,1);
\coordinate (5) at (1.25,1.9);
\coordinate (6) at (1.25,1.6);
\coordinate (7) at (1.25,1.15);
\coordinate (8) at (1.25,0.85);

\draw (0)--(2)--(3')--(1);
\draw (3)--(3');
\draw (2)--(4);
\draw (5)--(3)--(6);
\draw (8)--(4)--(7);

\draw[ball color=black] (2) circle (0.2mm);
\draw[ball color=black] (3) circle (0.2mm);
\draw[ball color=black] (3') circle (0.2mm);
\draw[ball color=black] (4) circle (0.2mm);

\node at (-2.25,1) {\tiny \boxed{$2$}};
\node at (-2.25,2) {\tiny \boxed{$1$}};
\node at (1.5,2) {\tiny \boxed{$4$}};
\node at (1.5,1.6) {\tiny \boxed{$3$}};
\node at (1.5,1.15) {\tiny \boxed{$5$}};
\node at (1.5,0.65) {\tiny \boxed{$6$}};

\node at (-1,2) {\tiny $4$};
\node at (-1.75,0.85) {\tiny $4$};
\node at (-0.75,1.5) {\tiny $2$};
\node at (-0.5, 0.85) {\tiny $2$};
\node at (1, 2) {\tiny $1$};
\node at (1, 1.5) {\tiny $5$};
\node at (1,1.2) {\tiny $1$};
\node at (1,0.8) {\tiny $1$};
\node at (0.25,1.85) {\tiny $6$};
\end{tikzpicture}
}
\caption{The cone spanned by $v_{\{3,4\}}$, $v_{\{1,3,4\}}$ and $v_{\{5,6\}}$ and its subdivision induced by stabilization.}
\label{fig: neighboring-cone}
\end{figure}

\end{example}


\subsection{Tropicalizations of moduli spaces}

In this section, we construct a tropicalization map from the analytification of the space of relative stable maps to the (combinatorial) tropical space. This will be a crucial step in our proof of tropical compactification. 

\subsubsection{Functorial tropicalization for $\overline M_{0,n}$} We recall first the tropicalization map
\[
trop: \overline M_{0,n}^{\an}\to \overline M_{0,n}^{\trop}.
\]
This map has been introduced and studied by many authors, but we follow the presentation and conventions of~\cite{ACP}.

A point of the Berkovich analytic space $\overline M_{0,n}^{\an}$ may be represented by a map 
\[
Spec(K)\to \overline M_{0,n},
\]
where $K$ is a valued field extending the ground field with valuation ring $R$, residue field $\kappa$, maximal ideal $\mathfrak m$, and valuation $val(-)$. Notice that $\overline M_{0,n}$ is a proper variety, and thus we obtain a map
\[
Spec(R)\to \overline M_{0,n}.
\]
By pulling back the universal family, we obtain a family of curves $[C\to Spec(R)]$. Let $\Gamma(C)$ be the marked dual graph of the special fiber of this family. In order to obtain a tropical curve, we must metrize the bounded edges of this graph. Let $e$ be an edge of $\Gamma(C)$ corresponding to a node $q\in C_{\kappa}$. The local defining equation of $q$ in the total family is given by
\[
xy = f,
\]
where $f\in \mathfrak m$. We set the length $\ell(e)$ of the edge $e$ to be $val(f)$. We thus obtain a tropical curve $\Gamma(C)$.

The following is a consequence of~{\cite[Theorem 1.2.1]{ACP}}.

\begin{theorem}\label{thm-functorialtropm0n}
The tropicalization map $trop: \overline M_{0,n}^{\an}\to \overline M_{0,n}^{\trop}$ is continuous, and is naturally identified with the projection from $\overline M_{0,n}^{\an}$ to the Thuillier skeleton associated to the toroidal embedding $M_{0,n}\hookrightarrow \overline M_{0,n}$. 
\end{theorem}

\subsubsection{Functorial tropicalization for rubber relative stable maps} We analogously define a tropicalization map as above on the space of relative stable maps,
\[
trop:\overline M^{\an}(\bm x)\to \overline M^{\trop}(\bm x). 
\]
Let $K$ be a valued field extending the base field, with valuation ring $R$, residue field $\kappa$, maximal ideal $\mathfrak m$, and valuation $val(-)$. A point of the analytic moduli space can be represented by a $K$-valued point
\[
Spec(K)\to \overline M(\bm x).
\]
By properness of the stack $\Mbar(\bm x)$, after a finite base extension $K'|K$, this extends uniquely to a point over $Spec(R')$, where $R'$ is the valuation ring of $K'$. Note that the valuation extends uniquely to any finite field extension of $K$. By pulling back the universal family and the universal map, we obtain a family of maps $[C\to T]$ over $Spec(R')$. Metrizing the dual graphs of the source and target, we obtain two metric graphs (with marked ends) $\Gamma(C)$ and $\Gamma(T)$.

Since $[\pi: C\to T]$ is a relative stable map we make the following observations. 
\begin{description}
\item[Vertices] Every irreducible component of $C$ maps to a unique irreducible component of $T$, and we obtain a map of vertices.
\item[Edges] If the image of a node is not a node, then the corresponding edge is contracted to a vertex. Else, if two components $C_1$ and $C_2$ of $C$ share a node, then their images share a node as well, and thus every edge of the source maps to an edge of the target.
\item[Expansion factors] Suppose that $\widetilde q$ is a node of $C$ mapping to $q\in T$. By the pre-deformability condition, if we lift $\pi$ to a map between the normalizations $\widetilde C$ and $\widetilde T$, the local degree on each shadow of the node $\widetilde q$ maps to a shadow of the node $q$ with some degree $d$. That is, the degrees on either shadow must be equal for $\pi$ to be deformable. Thus, we see that if the node $q$ has local equation $xy = f$, then the node $\widetilde q$ has local equation 
\[
uv = f^d,
\]
and thus, each edge has a well defined expansion factor $d$. Note that if $q$ is a smooth point, we may take $d = 0$.
\end{description}

We conclude that there is an induced map $[\Gamma(C)\to \Gamma(T)]$, which is a tropical relative stable map to $\PP^1_{\trop}$, up to additive $\RR$ action on the base, i.e. up to a choice of root vertex for the tree $\Gamma(T)$. 

\begin{proposition}
\normalfont
{\it (Functorial tropicalization for rubber target)} \label{prop-functorialtroprubber}
The map $trop:\Mbar^{\an}(\bm x)\to \overline M^{\trop}(\bm x)$ is a continuous projection. Moreover, 
\begin{enumerate}
\item there is an isomorphism of extended cone complexes with integral structure,
\[
\overline M^{\trop}(\bm x) \xrightarrow{\sim} \overline \Sigma(\Mbar^{}(\bm x));
\]
\item the map $trop$ is functorial for tautological branch and stabilization morphisms. That is, the following diagram commutes:

\[
\begin{tikzcd}
\overline M^{\an}(\bm x) \arrow{dr}{\trop} \arrow[swap]{dd}{br^{\an}}  \arrow{rr}{st^{\an}}  &   &   \overline M^{\an}_{0,n}  \arrow{d}{\trop}\\
    & \overline M^{\trop}(\bm x)\arrow{d}{\mathfrak{br}}  \arrow{r}{\mathfrak{st}}  &  \overline M^{\trop}_{0,n}\\
{[\overline M^{\an}_{0,2+\epsilon s}/\mathcal S_s]} \arrow[swap]{r}{\trop} & {[\overline M^{\trop}_{0,2+\epsilon s}/\mathcal S_s]}. & 
\end{tikzcd}
\]
\end{enumerate}
\end{proposition}

\begin{proof} We remark that such results have been carried out in various contexts, for moduli spaces of stable curves, weighted stable curves, and spaces of admissible covers in~\cite{ACP,CMR14,U14}. The Statement (2) above is closely related to the functoriality results for source and branch maps for tropicalizations of admissible covers~\cite{CMR14}. We outline the proof in our case. 

Note that the top dimensional strata of the tropical moduli space are in bijection with the zero strata of the space $\overline M(\bm x)$. Let $[C\to T]$ be a family of relative stable maps over $Spec(R)$, representing a point $P$ over a valuation ring $R$. In other words, $P$ is a point of $\Mbar^{\an}(\bm x)$. Let $[\Gamma(C)\to \Gamma(T)]$ be the map of dual graphs of special fibers. There exists a neighborhood of $P$ such that each node $q_i$ of $T$ has local defining equation $xy = f_i$. These functions $f_i$ cut out a divisor $D_i$ in a neighborhood of $P$, parametrizing those deformations of $[C\to T]$ where the $i$th node persists. These parameters $f_i$ thus form the basis for the monoid of effective Cartier divisors near $P$. That is, the $f_i$ yield a system of monomial coordinates near $P$. The valuations of these parameters are precisely the lengths assigned in the above construction, and coincide with the explicit description of Thuillier's retraction~\cite[Section 5.2]{ACP}. Thus, the tropicalization map coincides with the underlying set theoretic map of Thuillier's continuous projection, and the first result follows. 

Now consider the stabilization morphism $st: \overline M(\bm x) \to \overline M_{0,n}$. Consider a point $P\in \overline M(\bm x)$, representing the map $[C\to T]$. The universal family of $\overline M_{0,n}$ over the point $st(P)$ is the stabilization $C^{st}$ of the source curve $C$. Let $q$ be a node of $C^{st}$, and let $q_1,\ldots, q_{c}$ be the nodes of $C$ stabilizing to $q$. The pullback $st^\star(\xi_q)$ of the deformation parameter of $q$ is given by 
\[
st^\star(\xi_q) = \prod_{j = 1}^c \xi_{q_j},
\]
where $\xi_{q_j}$ is the deformation parameter of the node $q_j$. Thus, the map $st$ is locally formally given by monomials. It is clearly dominant in local charts, and thus, $st^{an}$ restricts to a map between skeletons. 

Finally, we must verify that the restriction of this analytic stabilization morphism coincides with tropical stabilization. This follows from identical arguments to~\cite[Proposition 8.2.4]{ACP}. The case of the branch map follows by similar arguments, c.f.~\cite[Section 5.2]{CMR14}.
\end{proof}

\subsubsection{Functorial tropicalization for parametrized target} A very similar argument to the above one can be used to prove the following result for parametrized target spaces. The tropicalization map is constructed analogously. A point of $\Mbar_r^{\an}(\PP^1,\bm x)$ is represented by a morphism from a rank $1$ valuation ring $R$
\[
Spec(R)\to \Mbar_r^{\an}(\PP^1,\bm x).
\]
Pulling back the universal family and the universal map, and taking metrized dual graphs as above, we obtain a map of graphs $\Gamma(C)\to \Gamma(T)$. The contraction morphism $\mathfrak c$ yields an identification of the main component on $\Gamma(T)$, which we denote as $0$. This data determines a tropical relative stable map with parametrized target $\PP^1_{\trop}$. We thus obtain a map 
\[
trop: \overline M_r^{\an}(\PP^1,\bm x)\to\overline M_r^{\trop}(\PP_{\trop}^1,\bm x).
\]
\begin{proposition}~\label{prop: ev-is-functorial}
The map $trop: \overline M_r^{\an}(\PP^1,\bm x)\to\overline M_r^{\trop}(\PP_{\trop}^1,\bm x)$ is a continuous projection, canonically identified with the projection to the Thuillier skeleton of $\overline M^{\an}(\PP^1,\bm x)$. Furthermore, tropicalization is compatible with evaluation morphisms. That is, the following diagram commutes
\[
\begin{tikzcd}
\Mbar^{\an}_r(\PP^1,\bm x)\arrow{r}{\trop} \arrow[swap]{d}{ev_i^{\an}} & \overline M^{\trop}_r(\PP^1_{\trop},\bm x) \arrow{d}{\mathfrak{ev}_i} \\
\PP^1_{\an}\arrow[swap]{r}{\trop} & \PP^1_{\trop}.\\
\end{tikzcd}
\]
\end{proposition}

\begin{proof}
The agreement of the Thuillier projection and set theoretic tropicalization maps is identical to that of the rubber case. The evaluation morphisms are the only new feature of the parametrized space. We want to check that the map $ev_i^{\an}$ restricts to a map between skeletons, and moreover, that this restricted map coincides with the tropical evaluation morphism. To check that $ev_i^{\an}$ restricts to skeletons, it is sufficient to show that $ev_i$ is locally analytically given by a dominant equivariant map of toric varieties. 

Let $P$ be a point of $\overline M(\PP^1,\bm x)$ represented by a relative stable map $[h:C\to T]$. Let $ev_i(h)$ be the image of the $i$th marked point on $C$. Moreover, assume that 
\[
T = T_{-k}\cup T_{-k+1}\cup \cdots\cup \boxed{T_0} \cup \cdots \cup T_{\ell},
\]
is the expanded target, and $T_0$ is the main component. For any representative $h: C\to T$, the image of the $i$th mark lies in the smooth locus of $T$, and thus we may assume it lies on the component $T_r$. Notice that even though the point $h(p_i)$ can depend on the choice of representative in the isomorphism class of $h$, $r$ is well defined. 

Choose a local toric chart near $P$, with (a subset of the) local monomial coordinates given by the deformation parameters $\xi_i$ of the nodes of $T$. On $\PP^1$, we work affine locally, choosing the toric chart near $0$, given by the coordinate $t$. The expanded target $T$ may be constructed from the trivial family $\PP^1\times \A^1$, by toric blowup in the special fiber. Thus, $T$ is isomorphic to the fiber over $0$ of $\widetilde T\to \A^1$, and there is a blow down map
\[
\widetilde T \to \PP^1\times \A^1.
\]
Restricting to the special fiber, we obtain the contraction map $\mathfrak{c}: T\to \PP^1$ to the main component. A standard local computation shows $ev_i^\star(t) = \prod_{0<i<r} \xi_i$. Thus, $ev_i$ induces a map which locally analytically pulls back monomials to monomials. The map is clearly dominant in each chart, and thus is toroidal. It is elementary to check that the restriction to the skeleton $\overline M^{\trop}_r(\PP^1_{\trop}, \bm x)$ coincides with the tropical evaluation morphism.
\end{proof} 

\subsection{Comparison with admissible covers} In~\cite{CMR14}, the authors study Berkovich skeletons of the spaces of admissible covers of curves of arbitrary genus. In the present case, where source and target have genus $0$, the spaces of admissible covers and rubber maps are closely related. Let $Adm(\bm x)$ be the space of admissible covers in the sense of Abramovich, Corti, and Vistoli~\cite{ACV}. There is a natural map obtained from admissible covers to rubber maps, obtained by forgetting the non-relative branch points, and stabilizing as necessary. We have a commutative diagram
\[
\begin{tikzcd}
Adm(\bm x)\arrow{r} \arrow{d} & \overline M(\bm x)\arrow{d} \\
\overline M_{0,n}\arrow{r} & {[\overline M_{0,2+s\cdot\epsilon}/S_s]}\\
\end{tikzcd}
\]

There is a natural forgetful morphism on the tropical side, 
\[
\Lambda^{\trop}: Adm^{\trop}(\bm x)\to \overline M^{\trop}(\bm x),
\]

contracting infinite edges with non-special branching orders, and all preimages. We recall~\cite[Theorem 1]{CMR14} that the set theoretic tropicalization factors as
\[
\begin{tikzcd}
Adm^{\an}(\bm x) \arrow{r}{\bm p} \arrow[bend left]{rr}{\trop} & \overline \Sigma(Adm(\bm x)) \arrow{r}{trop_\Sigma} & Adm^{\trop}(\bm x) \\
\end{tikzcd}
\]

where $\bm p$ is the projection to the skeleton, and $trop_\Sigma$ is an isomorphism upon restriction to any cone of the skeleton. The following corollary is a straightforward consequence of the discussion in the previous section. 

\begin{corollary}
The restriction of $\Lambda^{an}$ to the skeleton $\overline \Sigma(Adm(\bm x))$, composed with $trop_\Sigma$ coincides with the map $\Lambda^{\trop}$. That is, the following diagram commutes:
\[
\begin{tikzcd}
Adm^{\an}(\bm x) \arrow{rr}{\trop} \arrow[swap]{d}{\Lambda^{\an}} &  & Adm^{\trop}(\bm x) \arrow{d}{\Lambda^{\trop}}\\
\overline M^{\an}(\bm x) \arrow[swap]{rr}{\trop} & & \overline M^{\trop}(\bm x)\\ 
\end{tikzcd}
\]
\end{corollary}

\section{Tropical compactification}\label{sec: trop-compactification-rsm}

In this section, we explain how to deduce Theorem~\ref{thm: tropical-compactification} from the results in the preceding section. We focus first on the case of rubber target, and then indicate the necessary changes in the parametrized case. 

\subsection{The case of rubber target}

Let $U\hookrightarrow X$ be a toroidal embedding. A toroidal modification of $X$ is a new toroidal embedding $U\hookrightarrow X'$, together with a map 
\[
\pi: X'\to X,
\]
that is proper, birational, and toroidal. Recall that the toroidal morphism $\pi: X' \to X$ is a morphism for which there exist local toric charts on $X'$ and $X$, such that $\pi$ is given by a dominant equivariant map of toric varieties. There is an induced map on cone complexes $\Sigma(Y)\to \Sigma(X)$. Conversely, given a subdivision of $\widetilde \Sigma\to\Sigma(X)$, one constructs a unique associated toroidal modification $U\hookrightarrow Y$, such that $\Sigma(Y)\cong \widetilde \Sigma$. See~\cite[Section 1.4]{AK00} and~\cite[Chapter 2]{KKMSD} for further details. 

It follows from Proposition \ref{prop-functorialtroprubber} and Theorem \ref{thm-functorialtropm0n} that the tropical moduli spaces $ M^{\trop}(\bm x)$ resp.\ ${M}_{0,n}^{\trop}$ can be identified with the skeletons $\Sigma(\Mbar^{}(\bm x))$ resp.\ $\Sigma(\overline M_{0,n})$. By Proposition \ref{prop-subdivision}, the tropical stabilization morphism $\mathfrak{st}: \Sigma(\overline M(\bm x))\to \Sigma(\overline M_{0,n})$ is a subdivision. From the above, we obtain an induced toroidal modification, and a new toroidal compactification of $M_{0,n}$. 

\begin{proposition}\label{prop:int}
The toroidal modification induced by the  tropical stabilization morphism $\mathfrak{st}: \Sigma(\overline M(\bm x))\to \Sigma(\overline M_{0,n})$ is identified with $\overline M(\bm x)$. 
\end{proposition}

\begin{proof}
We first observe that the interior of $\overline M(\bm x)$ is identified with $M_{0,n}$. Given a map between smooth curves $[C\to \PP^1]$, we obtain a pointed rational curve $(C,p_1,\ldots, p_n)$ where $p_i$ is the $i$th special ramification point on the source curve $C$. Conversely, given $\bm x$ and the pointed curve $(C,p_1,\ldots, p_n)$, the divisor $\sum x_i p_i$ determines a rational function up to multiplicative constant. Since we are working with the rubber target, this determines the map $[C\to \PP^1]$. 

We know from Proposition~\ref{prop-functorialtroprubber} that $st$ induces a map on cone complexes, and furthermore that this map is a subdivision. Thus, $st$ is a toroidal, proper, birational map inducing a subdivision $\mathfrak{st}$. The result now follows from~\cite[Chapter II, Theorem 6]{KKMSD}.
\end{proof}

\subsection{Proof of Theorem~\ref{thm: tropical-compactification}} Let $M(\bm x)\subset \overline M(\bm x)$ be the locus of maps from smooth curves. Fix the identification $M(\bm x)$ with $M_{0,n}$, taking a map $[C\to \PP^1]$ to the source curve $C$, marked at the special ramification locus. 

By the discussion in the proof of Proposition \ref{prop:int}, we obtain an embedding of $M(\bm x)$ into a torus $T$ by producing an embedding of $M_{0,n}$ into $T$. We use the embedding due to Kapranov~\cite{Kap93}. Fix a vector space $V$ of dimension $n$, and let $T'$ be the $n$-dimensional torus that dilates the coordinates on $V$. There is a map
\[
M_{0,n}\hookrightarrow G^0(2,n)/\!/T',
\]
associating to $n$-points on $\PP^1$ a $2$-plane defined by the span of the $2\times n$ matrix containing homogeneous coordinates of $n$-points on $\PP^1$. The open set $G^0(2,n)$ is the intersection of the image of the Grassmannian  via the  Pl\"ucker embedding  with the torus of $\PP^{\binom{n}{2}-1}$.

The Pl\"ucker embedding now yields an embedding into the Chow quotient
\[
M_{0,n}\hookrightarrow \PP^{\binom{n}{2}-1}/\!/ T', 
\]

and the closure of the image of $M_{0,n}$ gives the stable curves compactification $\overline M_{0,n}$.

Denote the fan of this Chow quotient by $\mathcal F$. By work of Gibney and Maclagan~\cite{GM07}, the tropicalization of $M_{0,n}$ in this embedding equals the subfan $\Delta_n\subset \mathcal F$, where
\[
\Delta_n \cong M_{0,n}^{\trop},
\] 
see Page \ref{deltan}.
From Propisition \ref{prop-subdivision}, the morphism $\mathfrak{st}$ yields a simplicial subdivision $\Delta^{\rub}_{\bm x}$ of $\Delta_n$, and there is an associated proper, birational, toric morphism
\[
X(\Delta_{\bm x}^{\rub})\to X(\Delta_n). 
\]
Let $Y$ be the strict transform of $\overline M_{0,n}$ in $X(\Delta_{\bm x}^{\rub})$. $Y$ inherits the structure of a toroidal embedding from the toric boundary of $X(\Delta_{\bm x}^{\rub})$, and the induced map
\[
Y\to \overline M_{0,n}
\]
is a toroidal modification with cone complex $\overline M^{\trop}(\bm x)$. It now follows from the above proposition and~\cite[Section 2.2]{KKMSD} that $Y\cong \overline M(\bm x)$. To see that the compactification is sch\"on, observe that $\overline M_{0,n}$ is a sch\"on compactification of the space $M(\bm x)$ and $\overline M(\bm x)$ is a tropical compactification. By~\cite[Theorem 1.4]{Tev07} we may conclude that $\overline M(\bm x)$ is sch\"on.
 \qed

When there are $r$ additional non-relative marks on the source, the natural stabilization morphism takes values in $\overline M_{0,n+r}$. The same argument as above yields the following. 

\begin{theorem}
There is a simplicial (noncomplete) toric variety $X(\Delta_{\bm x,r}^{\rub})$ with dense torus $T^{\rub}_{\bm x, r}$, and an embedding $M_r(\bm x)\hookrightarrow T^{\rub}_{\bm x, r}$, such that the closure of $M_r(\bm x)$ in $X(\Delta_{\bm x,r}^{\rub})$ is identified with the coarse moduli space $\overline M_r(\bm x)$. This compactification of $M_r(\bm x)$ is sch\"on. Moreover, the cone complex underlying $\Delta^{\rub}_{\bm x, r}$ is naturally identified with the tropical moduli space $\overline M_r^{\trop}(\bm x)$. 
\end{theorem}

\subsection{The case of parametrized target} We now indicate the changes necessary when the target is parametrized. In order to rigidify the problem, we only consider the case where the source curve has at least $1$ non-relative marking. That is, we consider space $\Mbar_r(\PP^1,\bm x)$ for $r>0$. 

Consider the natural tautological morphism
\[
st\times ev_{n+1}: M_r(\PP^1,\bm x) \to M_{0,n+r}\times \mathbb G_m,
\]
where $\mathbb G_m$ is the dense torus in the target $\PP^1$ (compare with the tropical version described below Proposition \ref{prop-subdivision}). Note that if $f\in M_r(\PP^1,\bm x)$ then $ev_{n+1}(f)$ is a point in $\mathbb G_m$, since the total preimage of $0$ and $\infty$ are marked. 

\begin{lemma}
The map 
\[
st\times ev_{n+1}: M_r(\PP^1,\bm x) \to M_{0,n+r}\times \mathbb G_m,
\]
is an isomorphism.
\end{lemma}

\begin{proof}
Given a pointed curve $(C,p_1,\ldots, p_{n+r})$, the ramification profile $\bm x$ determines the map $C\to \PP^1$, with ramification $x_i$ over $p_i$, up to a constant multiple. The image of $p_{n+1}$ now fully determines the map.
\end{proof}

By arguments in the preceding sections, we see that
\[
st\times ev_{n+1}: \overline M_r(\PP^1,\bm x) \to \overline M_{0,n+r}\times \mathbb P^1
\]
is a proper, toroidal morphism, inducing a subdivision of cone complexes
\[
\mathfrak{st}\times \mathfrak{ev}_{n+1}: \Sigma(\overline M_r(\PP^1,\bm x))\to \Sigma(\overline M_{0,n+r})\times \Sigma(\PP^1).
\]
Here, $\Sigma(\overline M_r(\PP^1,\bm x))$ and $\Sigma(\overline M_{0,n+r})\times \Sigma(\PP^1)$ can
by Proposition \ref{prop: ev-is-functorial} and \ref{thm-functorialtropm0n} be identified with $\overline M^{\trop}_r(\PP^1_{\trop},\bm x)$ resp.\ $\overline M^{\trop}_{n+r}\times \PP^1_{\trop}$. Furthermore, the morphism $\mathfrak{st}\times \mathfrak{ev}_{n+1}$ can by the generalization of Proposition \ref{prop-subdivision} be viewed as a stellar subdivision of simplicial fans $$\mathfrak{st}\times \mathfrak{ev}_{n+1}: \Delta_{\bm x,r}^{\parm}\to \Delta_{n+r}\times  \PP^1_{\trop}.$$
Identical arguments as the previous section now produces the following parametrized version of the main theorem.

\begin{theorem}\label{thm-tropical-compactification-parametrized}
There is a simplicial (noncomplete) toric variety $X(\Delta_{\bm x,r}^{\parm})$ with dense torus $T^{\parm}_{\bm x,r}$, and an embedding $M_r(\PP^1,\bm x)\hookrightarrow T^{\parm}_{\bm x,r}$, such that the closure of $M_r(\PP^1,\bm x)$ in $X(\Delta_{\bm x,r}^{\parm})$ is identified with the coarse moduli space $\overline M_r(\PP^1,\bm x)$. This compactification of $M_r(\PP^1,\bm x)$ is sch\"on. Moreover, the cone complex underlying $\Delta^{\parm}_{\bm x,r}$ is naturally identified with the tropical moduli space of maps to a parametrized target $\overline M_r^{\trop}(\PP^1_{\trop}, \bm x)$. 
\end{theorem}

We have the following straightforward corollary.

\begin{corollary}
The evaluation map $ev_i: \overline M_r(\PP^1, \bm x)\to \PP^1$ extends to a morphism of toric varieties $\overline{ev}_i: X(\Delta^{\parm}_{\bm x,r})\to \PP^1$.
\end{corollary}

\begin{proof}
It is shown in~\cite[Proposition 3.6]{Gro14} that the evaluation morphism, restricted to the open locus,
\[
ev_i:M_r(\PP^1, \bm x)\to  \mathbb G_m\subset \PP^1,
\]
can be extended to a morphism of tori $\overline{ev}_i: T^{\parm}_{\bm x,r}\to  \mathbb G_m$ (where $T^{\parm}_{\bm x,r}$ denotes the torus we embed $M_r(\PP^1, \bm x)$ into, as in Theorem \ref{thm-tropical-compactification-parametrized}). To check that this map extends to a morphism of the corresponding toric varieties, we must verify that the image of every cone of $\Delta^{\parm}_{\bm x, r}$ lies in a cone of $\Sigma(\PP^1)$. Since $\mathfrak{ev_i}$ is a morphism of fans the result follows immediately. 
\end{proof}

\section{Applications}

In this section, we complete the proofs of functorial tropicalization for Hurwitz cycles and the descendant correspondence.

\subsection{Hurwitz loci}\label{sec: hurwitz-loci}

We continue to fix ramification data $\bm x$ with $n$ entries, and denote by $r$ the number of simple branch points for a generic cover of special ramification type $\bm x$, i.e.\ $r=n-2$. We now construct the (classical) $k$-dimensional Hurwitz loci, following~\cite{BCM, GV05}.

Consider the space of relative stable maps to a parametrized target $\overline M_{r}(\PP^1,\bm x)$, together with its natural forgetful morphism (the stabilization morphism followed by forgetting the $r$ additional marked points)
\[
ft: \overline M_r(\PP^1,\bm x)\to \overline M_{0,n}. 
\]
Let $\hat \psi_i$ denote the first Chern class of the $i$th non-relative cotangent line bundle, and $ev_i$ the $i$th non-relative evaluation morphism. Then we may define the class
\[
\mathbb{H}_k(\bm x) = ft_\star\left(\prod_{i=1}^{r-k} \hat \psi_i ev_i^\star(pt) \prod_{j = r-k+1}^r \hat \psi_j \right).
\]

Intuitively, each factor of $\psi$ is equivalent to forcing a simple ramification on the $i$th marked point~\cite{OP06}. Pairing $\hat \psi_i$ with an evaluation pullback forces the image of this simple ramification point to be fixed. In particular, when $k = 0$, the degree of $\mathbb H_0(\bm x)$ recovers the Hurwitz number. The interested reader can find the details in~\cite[Lemma 3.2]{BCM}. 

On the tropical side, we make the analogous definitions. Let $\bm x$ be ramification data, and $r = n-2$. Consider the moduli space of tropical maps $\overline M^{\trop}_{r}(\PP^1,\bm x)$, equipped with the forgetful morphism (stabilization followed by forgetting the $r$ additional marked points)
\[
\mathfrak{ft}:\overline M^{\trop}_{r}(\PP^1,\bm x)\to\overline M_{0,n}^{\trop}.
\]
The tropical Hurwitz cycle depends on the choice of $r-k$ points $p_1,\ldots, p_{r-k}\in \RR\subset \PP^1_{\trop}$. We define
\[
\mathbb H_k^{\trop}(p_1,\ldots, p_{r-k};\bm x) = \mathbb H_k^{\trop}(\underline p;\bm x):=\mathfrak{ft}_\star\left( \prod_{i = 1}^{r-k} \hat \psi^{\trop} ev_i^\star(p_i)\cdot \prod_{j = r-k+1}^r \hat \psi_j^{\trop}\right)\subset \overline M_{0,n}^{\trop}.
\]

\subsection{Evaluations and $\psi$-classes} In order to study descendant invariants and higher dimensional Hurwitz loci, we need certain facts about $\psi$ classes on the space of maps, and about  cycles arising as inverse images via the evaluation morphisms. 

\begin{proposition}\label{prop: ev-fibers}
Let $p\in \mathbb R\subset \PP^1_{\trop}$. Then the fiber, $\mathfrak{ev}_i^{-1}(p)$, is a balanced integral polyhedral complex in $|\Delta^{\parm}_{\bm x, r}|$. The cycles $\mathfrak{ev}_i^{-1}(p)$ for different choices of $p\in\mathbb R$ are (tropically) rationally equivalent. 
\end{proposition}

\begin{proof}
Fix a point $p\in \RR\subset \PP^1_{\trop}$. Recall that there is a natural inclusion map $\PP^1_{\trop}\hookrightarrow \PP^1_{\an}$, and thus we may think of $p$ as a point of $\PP^1_{\an}$. That is, $p$ is identified by a map $\mathscr M(K)\to \PP^1_{\an}$, where $\mathscr M(K)$ is the Berkovich spectrum of a valued field extension of $\CC$. The analytic fiber over $p$ is then described by pullback, 
\[
ev_{\an,i}^{-1}(p) = \overline M^{\an}_r(\PP^1,\bm x)\times_{\PP^1_{\an}}\mathscr M(K).
\]
From Proposition~\ref{prop: ev-is-functorial}, the tropical fiber $\mathfrak{ev}_i^{-1}(p)$ coincides with the image under $trop$ of $ev_{\an,i}^{-1}(p)$. To show that $\mathfrak{ev_i}^{-1}(p)$ is a balanced polyhedral complex, we exhibit it as the tropicalization of a subvariety of a toric variety~\cite[Theorem 3.3.6]{MS14}. Choose a field $L$ extending $\mathbb C$, such that there exists an $L$-rational point $x\in \PP^1(K)$ such that $trop(x) = p$. It follows from the above discussion and~\cite[Proposition 6.1]{Pay09} that  $trop(ev_i^{-1}(x))$, is precisely $\mathfrak{ev}_i^{-1}(p)$. If $p_1,p_2\in \RR$, choose points $x_1$ and $x_2$ on $\PP^1$ such that $trop(x_j) = p_j$. The fibers $ev_i^{-1}(x_1)$ and $ev_i^{-1}(x_2)$ are easily seen to be rationally equivalent, and the result follows.
\end{proof}

The cycles $\hat \psi_i^{\trop}$ are defined to be the locus of maps in $\overline M^{\trop}_r(\PP^1,\bm x)$ having a $4$-valent vertex incident to the $i$th marked end. On the other hand, recall that classically $\hat \psi_i$ may be represented by the locus of curves that are simply ramified at the $i$th marked point~\cite{OP06}. 

\begin{proposition}\label{prop: psi-cycles}
The locus $\hat \psi_i^{\trop}$ is the image under $trop$ of the cycle $Z_i\subset \Mbar^{\an}_r(\PP^1,\bm x)$ parametrizing stable maps which are simply ramified at the $i$th marked point.
\end{proposition}

\begin{proof}
The proof is essentially an unraveling of definitions. The inclusion $Z_i\hookrightarrow \Mbar_r(\PP^1,\bm x)$ produces an analytic map $Z_i^{\an}\hookrightarrow \Mbar^{\an}_r(\PP^1,\bm x)$ which may be described as follows. As we have seen previously, a point of $Z_i^{\an}$ is represented by a family of relative stable maps $[C\to T]$ over a germ of a curve $Spec(R)$, such that the generic fiber is simply ramified at the $i$th marked point. Composing the inclusion of $Z_i^{\an}$ with the tropicalization map, one easily observes that $trop(Z_i)$ coincides with the locus of covers having a $4$-valent vertex incident to the $i$th marked point. This is precisely the tropical cycle $\hat \psi_i$. 
\end{proof}

\begin{corollary}
The locus $\hat \psi_i^{\trop}$ is a balanced fan, with all weights equal to $1$.
\end{corollary}

\subsection{Proof of Theorem~\ref{thm: descendant-correspondence}} Recall that our moduli space $\overline M_r(\PP^1, \bm x)$ is a tropical compactification in the simplicial toric variety $X(\Delta^{\parm}_{\bm x,r})$. We can now find an open embedding $X(\Delta^{\parm}_{\bm x,r})$ into a complete simplicial toric variety $X(\widehat \Delta^{\parm}_{\bm x,r})$. Construct $\widehat \Delta^{\parm}_{\bm x,r}$ as follows. The tropical moduli space $\overline M^{\trop}_{0,n+r}$ is a collection of cones in the fan $\mathcal F$ of the Chow quotient $\PP^{\binom{n+r}{2}}/\!/T$, as described in Section~\ref{sec: trop-compactification-rsm}. Subdivide $\mathcal F\times \PP^1_{\trop}$ according to the subdivision $\mathfrak{st}\times \mathfrak{ev}_{n+1}$, producing a fan $\widetilde \Delta^{\parm}_{\bm x,r}$. Now, by the toric resolution of singularities algorithm, we subdivide to obtain a fan $\widehat \Delta^{\parm}_{\bm x,r}$, unimodular away from the cones of $\Delta_{\bm x,r}^{\parm}$. We thus obtain the desired embedding of $\overline M_r(\PP^1, \bm x)$ into a complete simplicial toric variety $X(\widehat \Delta^{\parm}_{\bm x,r})$. 

We may treat $\hat \psi_i$ and $ev_j^\star(pt)$ as elements of $A^1(X(\widehat \Delta^{\parm}_{\bm x,r}))$. As explained in~\cite[Section 7]{Kat09}, we may lift each of these elements to piecewise linear functions, i.e. elements of $A^1_T(X(\widehat \Delta^{\parm}_{\bm x,r}))$. These can now be extended, as  piecewise linear functions, to the complete toric variety $\widehat \Delta^{\parm}_{\bm x,r}$. To lift $ev_i^\star(pt)$, we simply choose any piecewise linear function (tropical Cartier divisor) on the fan of $\PP^1$ which bends with multiplicity $1$ at precisely the chosen point $p_i\in \RR\subset \PP^1_{\trop}$. Pulling this function back to $\Delta^{\parm}_{\bm x,r}$, we obtain the desired equivariant lifting of $ev_i^\star(pt)$. For the classes $\hat \psi_i$, we use the standard lift of $\psi_i$ described in~\cite{Kat09, KM09}. The descendant correspondence now follows from~\cite[Theorem 6.3]{Kat09}. \qed

\begin{example}
 In this example, we demonstrate how we can use Theorem~\ref{thm: descendant-correspondence} to compute relative descendant Gromov-Witten invariants of $\PP^1$, by computing the tropical relative stable maps  underlying the corresponding cycle in $A^1(X(\widehat \Delta^{\parm}_{\bm x,r}))$, with the multiplicity with which they appear in this cycle. 

Let us compute $\langle \tau_{2}(1_{\PP^1_{\trop}}),\tau_{1}(pt), \tau_{0}(pt)\rangle $ for the ramification data $\bm x=(-6,3,2,1)$, which we obtain by summing the multiplicities of the maps underlying the intersection $$(\hat \psi_5^{\trop})^2 \hat \psi_6^{\trop}\mathfrak{ev}_6^{\ast}(0)\mathfrak{ev}_7^{\ast}(1)$$ in $\overline M_3^{\trop}(\PP^1_{\trop},\bm x)$. 
Let us analyze these maps. The marked ends $6$ and $7$ are required to map to $0$ and $1$ in $\PP^1_{\trop}$, respectively. There is no condition where the marked end $5$ should map to, but since the intersection product is zero-dimensional, it has to be adjacent to one of the two vertices of the stabilization of the source curve in $\overline M^{\trop}_{0,7}$. 
Thus, end $5$ can either be adjacent to end $6$ and increase the valency of this vertex to $6=3+1+2 $, or it can be adjacent to $7$ and increase the valency to $5=3+0+2$. Remember by~\cite{KM09} the valence of a vertex adjacent to ends with markings $i_1,\ldots,i_r$ equals $3+k_{i_1}+\ldots+k_{i_r}$, where $k_{j}$ denotes the power of the Psi-class $ \hat \psi_j^{\trop}$ in the above intersection.
Figure~\ref{gw-example} shows the possible tropical stable maps that contribute to the computation. The multiplicities are computed according to~\cite{KM09} to be $\prod \frac{1}{k_i!}\prod_V (\sum_{i \mbox{ at } V} k_i)!$, where the second product goes over all vertices $V$ and the sum over the markings $i$ adjacent to $V$. Altogether, we obtain $\langle \tau_{2}(1_{\PP^1_{\trop}}),\tau_{1}(pt), \tau_{0}(pt)\rangle=12$.

\begin{figure}[h!]
\subfigure{
\begin{tikzpicture}
[scale=1.7]
\draw (-2,1)--(-0.5,1);
\draw (-0.5,1)--(1.25,1.5);
\draw (-0.5,1)--(1.25,1);
\draw (-0.5,1)--(1.25,0.5);
\draw[dashed] (-0.5,1)--(-0.7,1.5);
\draw[dashed] (-0.5,1)--(-0.3,1.5);
\draw[dashed] (0.375,1.25)--(.375,1.75);

\draw (-2,0)--(1.25,0);

\draw[ball color=black] (-0.5,0) circle (0.2mm);
\draw[ball color=black] (0.375,0) circle (0.2mm);

\node at (-2.2,1) {\tiny \boxed{$1$}};
\node at (1.45,1.5) {\tiny \boxed{$2$}};
\node at (1.45,1) {\tiny \boxed{$3$}};
\node at (1.45,.5) {\tiny \boxed{$4$}};

\node at (-0.7,1.7) {\tiny \boxed{$5$}};
\node at (-0.3,1.7) {\tiny \boxed{$6$}};
\node at (0.375,1.95) {\tiny \boxed{$7$}};

\node at (-1.5,1.1) {\tiny $6$};
\node at (0.65,1.45) {\tiny $3$};
\node at (0.65,1.1) {\tiny $2$};
\node at (0.65,.775) {\tiny $1$};

\end{tikzpicture}
}
\subfigure{
\begin{tikzpicture}
[scale=1.7]
\draw (-2,1)--(-0.5,1);
\draw (-.5,1)--(0.75,1.25);
\draw (-.5,1)--(1.25,0.5);
\draw (0.75,1.25)--(1.25,1.5);
\draw (0.75,1.25)--(1.25,1);

\draw[dashed] (-0.5,1)--(-0.5,1.5);
\draw[dashed] (0.75,1.25)--(.95,1.75);
\draw[dashed] (0.75,1.25)--(.55,1.75);

\draw (-2,0)--(1.25,0);

\draw[ball color=black] (-0.5,0) circle (0.2mm);
\draw[ball color=black] (0.75,0) circle (0.2mm);

\node at (-1.5,1.1) {\tiny $6$};
\node at (.25,1.25) {\tiny $5$};
\node at (.25,.7) {\tiny $1$};
\node at (1.15,1.55) {\tiny $3$};
\node at (1.15,.95) {\tiny $2$};

\node at (-2.2,1) {\tiny \boxed{$1$}};
\node at (1.45,1.5) {\tiny \boxed{$2$}};
\node at (1.45,1) {\tiny \boxed{$3$}};
\node at (1.45,.5) {\tiny \boxed{$4$}};

\node at (-0.5,1.7) {\tiny \boxed{$6$}};
\node at (0.525,1.95) {\tiny \boxed{$5$}};
\node at (1,1.95) {\tiny \boxed{$7$}};
\end{tikzpicture}
}
\subfigure{
\begin{tikzpicture}
[scale=1.7]
\draw (-2,1)--(-0.5,1);
\draw (-0.5,1)--(1.25,1.5);
\draw (-0.5,1)--(1.25,1);
\draw (-0.5,1)--(1.25,0.5);
\draw[dashed] (-0.5,1)--(-0.7,1.5);
\draw[dashed] (-0.5,1)--(-0.3,1.5);
\draw[dashed] (0.375,1.25)--(.375,1.75);

\draw (-2,0)--(1.25,0);

\draw[ball color=black] (-0.5,0) circle (0.2mm);
\draw[ball color=black] (0.375,0) circle (0.2mm);

\node at (-2.2,1) {\tiny \boxed{$1$}};
\node at (1.45,1.5) {\tiny \boxed{$3$}};
\node at (1.45,1) {\tiny \boxed{$2$}};
\node at (1.45,.5) {\tiny \boxed{$4$}};

\node at (-0.7,1.7) {\tiny \boxed{$5$}};
\node at (-0.3,1.7) {\tiny \boxed{$6$}};
\node at (0.375,1.95) {\tiny \boxed{$7$}};

\node at (-1.5,1.1) {\tiny $6$};
\node at (0.65,1.45) {\tiny $2$};
\node at (0.65,1.1) {\tiny $3$};
\node at (0.65,.775) {\tiny $1$};
\end{tikzpicture}
}
\subfigure{
\begin{tikzpicture}
[scale=1.7]
\draw (-2,1)--(-0.5,1);
\draw (-.5,1)--(0.75,1.25);
\draw (-.5,1)--(1.25,0.5);
\draw (0.75,1.25)--(1.25,1.5);
\draw (0.75,1.25)--(1.25,1);

\draw[dashed] (-0.5,1)--(-0.5,1.5);
\draw[dashed] (0.75,1.25)--(.95,1.75);
\draw[dashed] (0.75,1.25)--(.55,1.75);

\draw (-2,0)--(1.25,0);

\draw[ball color=black] (-0.5,0) circle (0.2mm);
\draw[ball color=black] (0.75,0) circle (0.2mm);

\node at (-1.5,1.1) {\tiny $6$};
\node at (.25,1.25) {\tiny $5$};
\node at (.25,.7) {\tiny $2$};
\node at (1.15,1.55) {\tiny $3$};
\node at (1.15,.95) {\tiny $1$};

\node at (-2.2,1) {\tiny \boxed{$1$}};
\node at (1.45,1.5) {\tiny \boxed{$2$}};
\node at (1.45,1) {\tiny \boxed{$4$}};
\node at (1.45,.5) {\tiny \boxed{$3$}};

\node at (-0.5,1.7) {\tiny \boxed{$6$}};
\node at (0.525,1.95) {\tiny \boxed{$5$}};
\node at (1,1.95) {\tiny \boxed{$7$}};
\end{tikzpicture}

}

\subfigure{
\begin{tikzpicture}
[scale=1.7]
\draw (-2,1)--(-0.5,1);
\draw (-0.5,1)--(1.25,1.5);
\draw (-0.5,1)--(1.25,1);
\draw (-0.5,1)--(1.25,0.5);
\draw[dashed] (-0.5,1)--(-0.7,1.5);
\draw[dashed] (-0.5,1)--(-0.3,1.5);
\draw[dashed] (0.375,1.25)--(.375,1.75);

\draw (-2,0)--(1.25,0);

\draw[ball color=black] (-0.5,0) circle (0.2mm);
\draw[ball color=black] (0.375,0) circle (0.2mm);

\node at (-2.2,1) {\tiny \boxed{$1$}};
\node at (1.45,1.5) {\tiny \boxed{$4$}};
\node at (1.45,1) {\tiny \boxed{$2$}};
\node at (1.45,.5) {\tiny \boxed{$3$}};

\node at (-0.7,1.7) {\tiny \boxed{$5$}};
\node at (-0.3,1.7) {\tiny \boxed{$6$}};
\node at (0.375,1.95) {\tiny \boxed{$7$}};

\node at (-1.5,1.1) {\tiny $6$};
\node at (0.65,1.45) {\tiny $1$};
\node at (0.65,1.1) {\tiny $3$};
\node at (0.65,.775) {\tiny $2$};
\end{tikzpicture}
}
\subfigure{
\begin{tikzpicture}
[scale=1.7]
\draw (-2,1)--(-0.5,1);
\draw (-.5,1)--(0.75,1.25);
\draw (-.5,1)--(1.25,0.5);
\draw (0.75,1.25)--(1.25,1.5);
\draw (0.75,1.25)--(1.25,1);

\draw[dashed] (-0.5,1)--(-0.5,1.5);
\draw[dashed] (0.75,1.25)--(.95,1.75);
\draw[dashed] (0.75,1.25)--(.55,1.75);

\draw (-2,0)--(1.25,0);

\draw[ball color=black] (-0.5,0) circle (0.2mm);
\draw[ball color=black] (0.75,0) circle (0.2mm);

\node at (-1.5,1.1) {\tiny $6$};
\node at (.25,1.25) {\tiny $5$};
\node at (.25,.7) {\tiny $3$};
\node at (1.15,1.55) {\tiny $2$};
\node at (1.15,.95) {\tiny $1$};

\node at (-2.2,1) {\tiny \boxed{$1$}};
\node at (1.45,1.5) {\tiny \boxed{$3$}};
\node at (1.45,1) {\tiny \boxed{$4$}};
\node at (1.45,.5) {\tiny \boxed{$2$}};

\node at (-0.5,1.7) {\tiny \boxed{$6$}};
\node at (0.525,1.95) {\tiny \boxed{$5$}};
\node at (1,1.95) {\tiny \boxed{$7$}};
\end{tikzpicture}
}
\caption{The combinatorial types contributing to the Gromov--Witten invariant $\langle \tau_2(\PP^1_{\trop}),\tau_1(pt),\tau_0(pt) \rangle$. The types on the left each have multiplicity $3$, while the types on the right each have multiplicity $1$.}
\label{gw-example}
\end{figure}
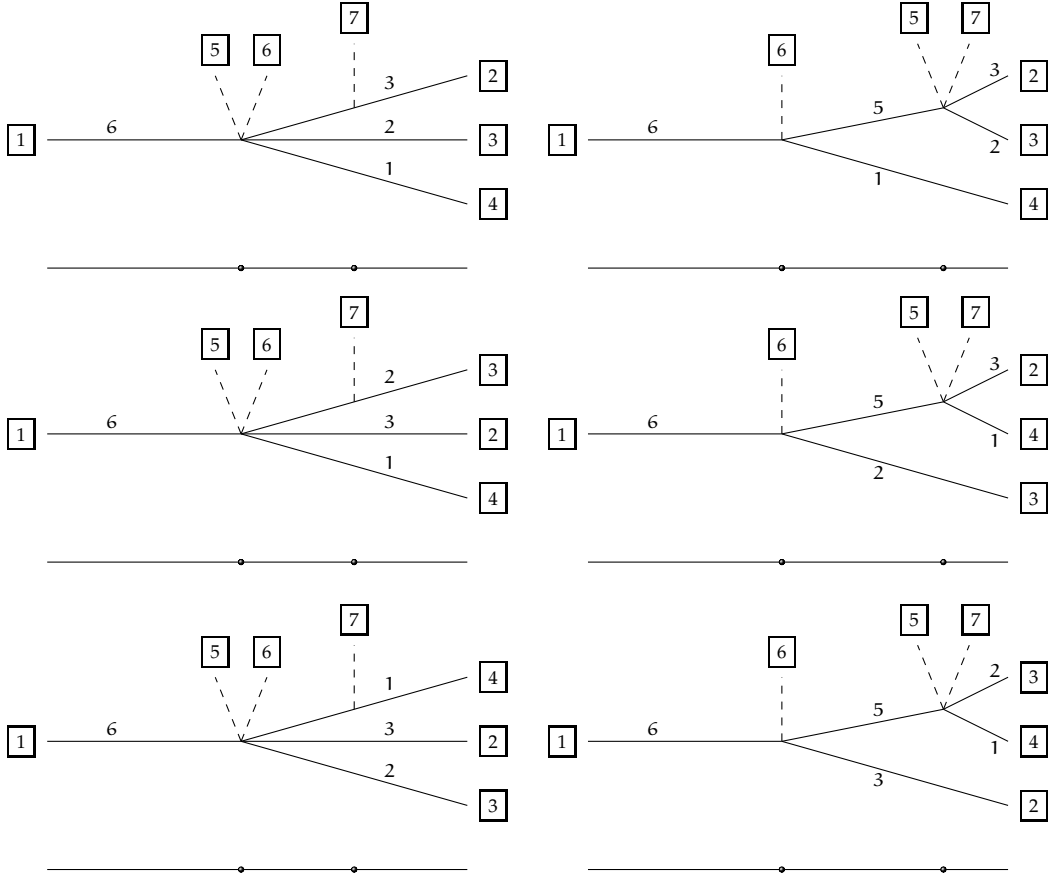
\end{example}

The example can easily be generalized to an algorithm to compute tropical relative descendant invariants of $\PP^1$.

\subsection{Proof of Theorem~\ref{thm: hurwitz-correspodnence}} Let $r = n-2$ denote the number of simple branch points on a generic cover of $\PP^1$ with ramification conditions given by $\bm x$. Fix a collection of (tropical) branch points  $\underline p = (p_1,\ldots, p_{r-k})\in (\RR)^{r-k}$. Consider the tropical Hurwitz locus $\mathbb H_k^{\trop}(\underline p;\bm x)$. Let $K$ be an algebraically closed valued field, extending the ground field $\CC$, such that the valuation $val:K^*\to \RR$ is surjective. Since tropicalization is invariant under field extension~\cite[Appendix]{Pay09}, we may base change to $\PP^1(K)$ and $\overline M_r(\PP^1,\bm x)(K)$, to obtain surjective tropicalization maps to $\PP^1_{\trop}$ and $\overline M^{\trop}_r(\PP^1,\bm x)$ respectively. Now, choose a collection of points $\underline y = (y_1,\ldots, y_{r-k})\in (\mathbb G_m(K))^{r-k}$ such that $trop(\underline y) = \underline p$.  Fix evaluation pullbacks $ev_i^{-1}(y_i)$, as subvarieties of the $K$-variety $\overline M_r(\PP^1,\bm x)(K)$. It follows from Proposition~\ref{prop: ev-fibers}, that $\mathfrak{ev}_i^{-1}(p_i)$ is the tropicalization of $ev_i^{-1}(y_i)$. Similarly, the class $\hat \psi_i$ is represented by the locus of stable maps over $K$ that are simply ramified at the $i$th marked point. Proposition~\ref{prop: psi-cycles} ensures that the class $\hat \psi_i^{\trop}$ is the tropicalization of this representative of $\hat \psi_i$ on the space of maps. We define the locus $\widetilde H_k(\underline y;\bm x)\subset \overline M_r(\PP^1,\bm x)(K)$, to be the the locus of maps where the $i$th marked point maps to the $p_i$ for $1\leq i \leq r-k$, and the stable map is simply ramified at all non-relative markings. The image of this locus in $M_{0,n}(K)$ yields a subvariety $H_k(\underline y;\bm x)$, such that $trop(H_k(\underline y;\bm x)) = \mathbb H_k^{\trop}(\underline p;\bm x)$. \qed
\newpage

\section*{Index of notation}

\begin{table}[h!]
\begin{tabular}{p{25mm}p{87mm}p{38mm}}
\hline
$X(\Delta)$ & the toric variety associated to a fan $\Delta$ & Section \ref{sec-prelim} \\
$X^{\an}$ & the Berkovich analytification of a variety $X$& Section \ref{sec-prelim}\\
$\Sigma(X)$& the Thuillier skeleton of a variety $X$ &  Section \ref{sec-prelim}\\
$\bm x$ & the ramification data above $0$ and $\infty$ --- a tuple of integers with $n$ entries, summing up to zero & Page \pageref{x}\\
$\Mbar_r(\PP^1,\bm x)$& the space of relative stable maps with ramification $\bm x$ and $r$ additional marked points to a parametrized target & Subsection \ref{subsec-classicalparam} and \ref{subsec-classicalmoduli}\\
$\Mbar_r(\bm x)$& the space of relative stable maps with ramification $\bm x$ and $r$ additional marked points to a rubber target & Subsection \ref{subsec-classicalrubber} and \ref{subsec-classicalmoduli}\\
$\Mbar(\PP^1,\bm x)$& the space of relative stable maps with ramification $\bm x$ (no additional marked points) to a parametrized target & Subsection \ref{subsec-classicalmoduli}\\
$\Mbar(\bm x)$& the space of relative stable maps with ramification $\bm x$ (no additional marked points) to a rubber target & Subsection \ref{subsec-classicalmoduli}\\
$ev_i$ & the evaluation morphism (for parametrized target, and for one of the $r$ additional marked points) & Page \pageref{tautologicalmorphisms}\\
$st$ & the stabilization morphism (the source map) & Page \pageref{tautologicalmorphisms} \\
$br$ & the branch map & Page \pageref{tautologicalmorphisms}\\
$\overline M_r^{\trop}(\PP^1_{\trop},\bm x)$& the space of tropical relative stable maps with ramification $\bm x$ and $r$ additional marked points to a parametrized target & Definition \ref{def-trop-parametrized}, Remark \ref{rem-trop-parametrized} and Subsection \ref{subsec-tropicalmoduli}\\
$\overline M_r^{\trop}(\bm x)$& the space of tropical relative stable maps with ramification $\bm x$ and $r$ additional marked points to a rubber target & Definition \ref{def-trop-rubber} and Subsection \ref{subsec-tropicalmoduli}\\
$\overline M^{\trop}(\PP^1_{\trop},\bm x)$& the space of tropical relative stable maps with ramification $\bm x$ (no additional marked points) to a parametrized target & Subsection \ref{subsec-tropicalmoduli}\\
$\overline M^{\trop}(\bm x)$& the space of tropical relative stable maps with ramification $\bm x$ (no additional marked points) to a rubber target & Subsection \ref{subsec-tropicalmoduli}\\
$\mathfrak{ev}_i$ & the tropical evaluation morphism (for parametrized target, and for one of the $r$ additional marked points) & Subsection \ref{subsec-tropicalmoduli}\\
$\mathfrak{st}$ & the tropical stabilization morphism (source map) & Subsection \ref{subsec-tropicalmoduli}\\
$\mathfrak{br}$ & the tropical branch map & Subsection \ref{subsec-tropicalmoduli}\\
$\Delta_n$ & the fan for $\overline M^{\trop}_{0,n}$& Page \pageref{deltan}\\
$\Delta^{\rub}_{\bm x (,r)}$& the subdivision of $\Delta_n$ induced by the stabilization morphism from the space of tropical relative stable maps to a rubber target, with ramification data $\bm x$ (and additional $r$ marked points) & Proposition \ref{prop-subdivision} \\
$\Delta^{\parm}_{\bm x ,r}$& the subdivision of $\Delta_n$ induced by the stabilization morphism times $\mathfrak{ev}_{n+1}$ from the space of tropical relative stable maps to a parametrized target, with ramification data $\bm x$ and additional $r$ marked points & Proposition \ref{prop-subdivision} \\

\hline
\end{tabular}
\end{table}

\newpage 
\bibliographystyle{siam}
\bibliography{GeomTropRSM}

\begin{thebibliography}{10}

\bibitem{ACP}
{\sc D.~{Abramovich}, L.~{Caporaso}, and S.~{Payne}}, {\em The tropicalization
  of the moduli space of curves}, Ann. Sci. \'Ec. Norm. Sup\'er., 48 (2015),
  pp.~765--809.

\bibitem{AC11}
{\sc D.~Abramovich and Q.~Chen}, {\em {Stable logarithmic maps to
  Deligne-Faltings pairs. II.}}, Asian J. Math, 18 (2014), pp.~465--488.

\bibitem{ACMUW}
{\sc D.~Abramovich, Q.~Chen, S.~Marcus, M.~Ulirsch, and J.~Wise}, {\em
  Skeletons and fans of logarithmic structures}, in Nonarchimedean and Tropical
  Geometry, M.~Baker and S.~Payne, eds., (To appear), p.~arXiv:1503.04343.

\bibitem{ACV}
{\sc D.~Abramovich, A.~Corti, and A.~Vistoli}, {\em Twisted bundles and
  admissible covers}, Comm. Alg., 31 (2003), pp.~3547--3618.

\bibitem{AK00}
{\sc D.~Abramovich and K.~Karu}, {\em Weak semistable reduction in
  characteristic 0}, Invent. Math., 139 (2000), pp.~241--273.

\bibitem{AMW12}
{\sc D.~Abramovich, S.~Marcus, and J.~Wise}, {\em Comparison theorems for
  {G}romov-{W}itten invariants of smooth pairs and of degenerations}, Ann.
  Inst. Fourier, 64 (2014), pp.~1611--1667.

\bibitem{BCM}
{\sc A.~Bertram, R.~Cavalieri, and H.~Markwig}, {\em {Polynomiality, wall
  crossings and tropical geometry of rational double Hurwitz cycles.}}, J.
  Comb. Theory, Ser. A, 120 (2013), pp.~1604--1631.

\bibitem{BBM}
{\sc B.~Bertrand, E.~Brugall{\'e}, and G.~Mikhalkin}, {\em Tropical open
  {H}urwitz numbers}, Rend. Semin. Mat. Univ. Padova, 125 (2011), pp.~157--171.

\bibitem{BHV01}
{\sc L.~Billera, S.~Holmes, and K.~Vogtmann}, {\em Geometry of the space of
  phylogenetic trees}, Adv. Appl. Math., 27 (2001), pp.~733--767.

\bibitem{CHMR14}
{\sc R.~Cavalieri, S.~Hampe, H.~Markwig, and D.~Ranganathan}, {\em {Moduli
  spaces of rational weighted stable curves and tropical geometry.}}, {Forum
  Math. Sigma}, 4 (2016), p.~e9.

\bibitem{CJM1}
{\sc R.~Cavalieri, P.~Johnson, and H.~Markwig}, {\em Tropical {H}urwitz
  numbers}, J. Alg. Comb., 32 (2010), pp.~241--265.

\bibitem{CMR14}
{\sc R.~Cavalieri, H.~Markwig, and D.~Ranganathan}, {\em Tropicalizing the
  space of admissible covers}, Math. Ann., 364 (2016), pp.~1275--1313.

\bibitem{Che10}
{\sc Q.~{Chen}}, {\em {Stable logarithmic maps to Deligne-Faltings pairs I.}},
  Ann. Math., 180 (2014), pp.~455--521.

\bibitem{Ful93}
{\sc W.~Fulton}, {\em {Introduction to toric varieties. The 1989 William H.
  Roever lectures in geometry.}}, Princeton, NJ: Princeton University Press,
  1993.

\bibitem{FS97}
{\sc W.~Fulton and B.~Sturmfels}, {\em Intersection theory on toric varieties},
  Topology, 36 (1997), pp.~335--353.

\bibitem{GKM07}
{\sc A.~Gathmann, M.~Kerber, and H.~Markwig}, {\em Tropical fans and the moduli
  space of rational tropical curves}, Comp. Math., 145 (2009), pp.~173--195.

\bibitem{GM07}
{\sc A.~Gibney and D.~Maclagan}, {\em Equations for {Chow} and {Hilbert
  Quotients}}, Algebra and Number Theory Journal, 4 (2010), pp.~855--885.
\newblock arXiv:0707.1801.

\bibitem{GJV}
{\sc I.~Goulden, D.~Jackson, and R.~Vakil}, {\em Towards the geometry of double
  {H}urwitz numbers}, Advances in Mathematics, 198 (2005), pp.~43--92.

\bibitem{GV05}
{\sc T.~Graber and R.~Vakil}, {\em Relative virtual localization and vanishing
  of tautological classes on moduli spaces of curves}, Duke Math. J., 130
  (2005), pp.~1--37.

\bibitem{Gro14}
{\sc A.~Gross}, {\em Correspondence theorems via tropicalizations of moduli
  spaces}, arXiv:1401.4626,  (2014).

\bibitem{Gro15}
\leavevmode\vrule height 2pt depth -1.6pt width 23pt, {\em Intersection theory
  on tropicalizations of toroidal embeddings}, arXiv:1510.04604,  (2015).

\bibitem{Gro10}
{\sc M.~Gross}, {\em {Mirror Symmetry for $\mathbb{P}^2$ and tropical
  geometry}}, Adv. Math., 224 (2010), pp.~169--245.

\bibitem{GS13}
{\sc M.~Gross and B.~Siebert}, {\em {Logarithmic Gromov-Witten invariants}}, J.
  Amer. Math. Soc., 26 (2013), pp.~451--510.

\bibitem{Gub13}
{\sc W.~Gubler}, {\em A guide to tropicalizations}, in Algebraic and
  combinatorial aspects of tropical geometry, vol.~589 of Contemp. Math., Amer.
  Math. Soc., Providence, RI, 2013, pp.~125--189.

\bibitem{H14}
{\sc S.~Hampe}, {\em {Combinatorics of tropical Hurwitz cycles}},
  arXiv:1407.3933,  (2014).

\bibitem{Kap93}
{\sc M.~Kapranov}, {\em {Chow quotients of Grassmannians {I}}}, in {IM
  Gel′fand Seminar}, vol.~16, 1993, pp.~29--110.

\bibitem{Kat09}
{\sc E.~Katz}, {\em Tropical intersection theory from toric varieties},
  Collect. Math., 63 (2012), pp.~29--44.
\newblock arXiv:0907.2488.

\bibitem{KKMSD}
{\sc G.~Kempf, F.~Knudsen, D.~Mumford, and B.~Saint-Donat}, {\em Toroidal
  embeddings {I}}, Lecture Notes in Mathematics, 339 (1973).

\bibitem{KM09}
{\sc M.~Kerber and H.~Markwig}, {\em Intersecting {Psi-classes} on tropical
  {$M_{0,n}$}}, Int. Math. Res. Not., 2009 (2009), pp.~221--240.

\bibitem{MS14}
{\sc D.~Maclagan and B.~Sturmfels}, {\em Introduction to tropical geometry},
  vol.~161, AMS Graduate Studies in Mathematics, 2015.

\bibitem{MR08}
{\sc H.~Markwig and J.~Rau}, {\em {Tropical descendant Gromov-Witten
  invariants}}, Manuscripta Mathematica, 129 (2009), pp.~293--335.

\bibitem{MP06}
{\sc D.~Maulik and R.~Pandharipande}, {\em A topological view of
  {G}romov--{W}itten theory}, Topology, 45 (2006), pp.~887--918.

\bibitem{Mi03}
{\sc G.~Mikhalkin}, {\em Enumerative tropical geometry in $\mathbb{R}^2$}, J.
  Amer. Math. Soc., 18 (2005), pp.~313--377.

\bibitem{Mi07}
\leavevmode\vrule height 2pt depth -1.6pt width 23pt, {\em Moduli spaces of
  rational tropical curves},  (2007), pp.~39--51.

\bibitem{OP06}
{\sc A.~Okounkov and R.~Pandharipande}, {\em Gromov-{W}itten theory, {H}urwitz
  theory, and completed cycles}, Ann. Math., 163 (2006), pp.~517--560.

\bibitem{OP}
{\sc B.~Osserman and S.~Payne}, {\em Lifting tropical intersections}, Documenta
  Mathematica, 18 (2013), pp.~121--175.

\bibitem{OR}
{\sc B.~Osserman and J.~Rabinoff}, {\em Lifting non-proper tropical
  intersections}, in Tropical and Non-Archimedean Geometry, O.~Amini, M.~Baker,
  and X.~Faber, eds., vol.~605, 2011, pp.~15--44.

\bibitem{Over15}
{\sc D.~P. Overholser}, {\em Descendent tropical mirror symmetry for
  $\mathbb{P}^2$}, arXiv:1504.06138,  (2015).

\bibitem{Pay09}
{\sc S.~Payne}, {\em Analytification is the limit of all tropicalizations},
  Math. Res. Lett., 16 (2009), pp.~543--556.

\bibitem{Pay15}
{\sc S.~Payne}, {\em Topology of nonarchimedean analytic spaces and relations
  to complex algebraic geometry}, Bull. Amer. Math. Soc., 52 (2015),
  pp.~223--247.

\bibitem{R15b}
{\sc D.~Ranganathan}, {\em Moduli of rational curves in toric varieties and
  non-archimedean geometry}, arXiv:1506.03754,  (2015).

\bibitem{SS04a}
{\sc D.~Speyer and B.~Sturmfels}, {\em The tropical {Grassmannian}}, Advances
  in Geometry, 4 (2004), pp.~389--411.

\bibitem{Tev07}
{\sc J.~Tevelev}, {\em Compactifications of subvarieties of tori}, American J.
  Math., 129 (2007), pp.~1087--1104.

\bibitem{Thu07}
{\sc A.~Thuillier}, {\em {G{\'e}om{\'e}trie toro{\"\i}dale et g{\'e}om{\'e}trie
  analytique non archim{\'e}dienne. Application au type d'homotopie de certains
  sch{\'e}mas formels}}, Manuscripta Mathematica, 123 (2007), pp.~381--451.

\bibitem{U13}
{\sc M.~Ulirsch}, {\em Functorial tropicalization of logarithmic schemes: The
  case of constant coefficients}, arXiv:1310.6269,  (2013).

\bibitem{U14}
\leavevmode\vrule height 2pt depth -1.6pt width 23pt, {\em Tropical geometry of
  moduli spaces of weighted stable curves}, J. Lond. Math. Soc., 92 (2015),
  pp.~427--450.

\bibitem{Vak08}
{\sc R.~Vakil}, {\em The moduli space of curves and {G}romov--{W}itten theory},
  in Enumerative invariants in algebraic geometry and string theory, Springer,
  2008, pp.~143--198.

\end{thebibliography}

\end{document}